\documentclass[12pt,reqno]{amsart}
\usepackage[a4paper,top=3cm,bottom=3cm,inner=2.5cm,outer=2.5cm]{geometry}

\usepackage{amsmath,amsfonts,amssymb,amsthm}
\usepackage{xparse}
\usepackage[all]{xy}
\usepackage{tikz}
\usepackage{tikz-cd}
\usepackage{multirow,enumitem,array}
\usepackage{pst-node,placeins,xspace,fix-cm}
\usepackage[nomessages]{fp}
\usepackage[colorlinks,linkcolor=blue,citecolor=violet,urlcolor=darkgray,final,hyperindex,linktoc=page,hyperfootnotes=true]{hyperref}

\newcolumntype{F}{>{$}c<{\hspace{-0.9ex}$}}
\newcolumntype{:}{>{$}m{0.8ex}<{$}}
\newcolumntype{R}{>{$}r<{$}}
\newcolumntype{C}{>{$}c<{$}}
\newcolumntype{L}{>{$}l<{$}}
\newcolumntype{N}{@{}>{$}l<{$}}
\newlength\horspace
\newcommand{\h}[1][1.0]{\hspace*{#1\horspace}}
\newlength\verspace
\tikzset{iso/.style={draw=none,every to/.append style={edge node={node [sloped, allow upside down, auto=false]{$\cong$}}}}}
\tikzset{adjunction/.style={draw=none,every to/.append style={edge node={node [sloped, allow upside down, auto=false]{$\dashv$}}}}}
\tikzset{adj/.style={draw=none,every to/.append style={edge node={node [sloped, allow upside down, auto=false]{$\dashv$}}}}}
\tikzset{simeq/.style={draw=none,every to/.append style={edge node={node [sloped, allow upside down, auto=false]{$\simeq$}}}}}
\tikzset{simeqS/.style={draw=none,every to/.append style={edge node={node [sloped, allow upside down, auto=false]{$\raisebox{0.8em}{$\simeq$}$}}}}}
\tikzset{aiso/.style={simeqS,preaction={draw,->}}}
\tikzset{proarrowS/.style={draw=none,every to/.append style={edge node={node [sloped, allow upside down, auto=false]{\raisebox{1.4pt}{\small$\shortmid$}}}}}}
\tikzset{proarrow/.style={proarrowS,preaction={draw,->}}}
\tikzset{nullS/.style={draw=none,every to/.append style={edge node={node [sloped, allow upside down, auto=false]{\raisebox{-1.16 ex}{$\circ$}}}}}}
\tikzset{null/.style={nullS,preaction={draw,->}}}
\tikzset{dotdot/.style={dash pattern=on 0.25ex off 0.2ex, dash phase=0ex}}
\tikzset{RightA/.style={double distance=3.5pt,>={Implies},->},%
	triple/.style={-,preaction={draw,RightA}},%
	quadruple/.style={preaction={draw,RightA,shorten >=0pt},shorten >=1pt,-,double,double distance=0.2pt}}
\tikzset{Right/.style={double distance=1.7pt,>={Implies},->}}
\tikzset{simeqSRight/.style={draw=none,every to/.append style={edge node={node [sloped, allow upside down, auto=false]{$\raisebox{-1em}{\rotatebox{180}{$\simeq$}}$}}}}}
\tikzset{twoiso/.style={simeqSRight,preaction={draw,Right}}}

\theoremstyle{plain}
\newtheorem{theorem}{Theorem}[section]
\newtheorem{lemma}[theorem]{Lemma}
\newtheorem{proposition}[theorem]{Proposition}
\newtheorem{prop}[theorem]{Proposition}

\theoremstyle{definition}
\newtheorem{definition}[theorem]{Definition}
\newtheorem{remark}[theorem]{Remark}
\newtheorem{example}[theorem]{Example}
\newtheorem{construction}[theorem]{Construction}
\newtheorem{notation}[theorem]{Notation}

\def\nameit#1{\textrm{#1}~}
\def\thex{\nameit{Theorem}}
\def\prox{\nameit{Proposition}}

\def\lemx{\nameit{Lemma}}

\def\remx{\nameit{Remark}}

\def\conx{\nameit{Construction}}
\def\dfn#1{{\itshape #1}}

\NewDocumentEnvironment{cd}{s O{6} O{6} b}{%
	\IfBooleanF{#1}{\begin{equation*}}\begin{tikzcd}[row sep=#2ex,column sep=#3ex,ampersand replacement=\&]
			#4
		\end{tikzcd}\IfBooleanF{#1}{\end{equation*}}\ignorespacesafterend}{}

\newenvironment{fun}{\[\begin{tabular}{F:RCL}}{\end{tabular}\]\ignorespacesafterend}

\newenvironment{eqD*}{\begin{equation*}}{\end{equation*}\ignorespacesafterend}

\newcommand{\C}{\mathbb{C}}
\newcommand{\D}{\mathbb{D}}

\newcommand{\E}{\mathcal{E}}
\newcommand{\M}{\mathcal{M}}

\newcommand{\N}{\mathcal{N}}
\newcommand{\T}{\mathcal{T}}
\newcommand{\F}{\mathcal{F}}

\newcommand{\Ab}{\mathbf{Ab}}

\newcommand{\Ses}[1]{\operatorname{Ses}\left({#1}\right)}
\newcommand{\ses}[1]{\Ses{#1}}

\newcommand{\BiHPTors}{\mathbf{BiHPTors}}

\NewDocumentCommand{\Alg}{t+ t' m}{
	\ensuremath{\IfBooleanT{#1}{\mathbf{Ps}\mbox{-}}{#3}\mbox{-}\mathbf{\IfBooleanT{#2}{Co}Alg}}
}
\newcommand{\ClIdl}{\mathbf{ClIdl}}
\newcommand{\ClIdlex}{\mathbf{ClIdl}^{\operatorname{ex}}}
\newcommand{\Oex}{\Omega^{\operatorname{ex}}}
\def\:{\colon}

\def\c{\circ}
\newcommand{\iso}{\cong}
\def\phi{\varphi}

\newcommand{\cont}{\subseteq}

\newcommand{\id}[1]{\operatorname{id}_{#1}}
\newcommand{\Id}[1]{\operatorname{Id}_{#1}}

\newcommand{\x}[1][]{\times_{#1}}
\newcommand{\too}{\longrightarrow}
\newcommand{\mto}{\mapsto}
\renewcommand{\ar}[2][]{\xrightarrow[#1]{#2}}

\makeatletter
\newcommand{\aR}[2][]{%
	\ext@arrow 0055{\Rightarrowfill@}{#1}{#2}}
\def\xLongrightarrowfill@{\arrowfill@\Relbar\Relbar\Longrightarrow}
\newcommand{\am}[2][]{%
	\ext@arrow 0395\xmapstofill@{#1}{#2}}
\def\xlongmapstofill@{\arrowfill@\relbar\relbar\longmapsto}
\def\xlongrightarrowfill@{\arrowfill@\relbar\relbar\longrightarrow}
\newcommand{\aarr}[2][]{%
	\ext@arrow 0099\xlongrightarrowfill@{#1}{#2}}
\newcommand{\eqq}{\DOTSB\protect\Relbar\protect\joinrel\Relbar}
\def\xeqqfill@{\arrowfill@\Relbar\Relbar\eqq}
\newcommand{\aeqq}[2][]{%
	\ext@arrow 0099\xeqqfill@{#1}{#2}}
\def\xRrightarrowfill@{\arrowfill@\equiv\equiv\Rrightarrow}
\newcommand{\aM}[2][]{\ext@arrow 0359\xRrightarrowfill@{#1}{#2}}    
\def\nullrightarrowfill@{\arrowfill@{\relbar{\circ}}\relbar\longrightarrow}
\newcommand{\anull}[2][]{
	\ext@arrow0099\nullrightarrowfill@{#1}{#2}}
\makeatother
\newcommand{\aiso}[1]{\overset{#1}{\iso}}

\newcommand{\aequi}{\ensuremath{\stackrel{\raisebox{-1ex}{\kern-.3ex$\scriptstyle\sim$}}{\rightarrow}}}
\newcommand{\aequii}{\ensuremath{\stackrel{\raisebox{-1ex}{\kern-.3ex$\scriptstyle\sim$}}{\longrightarrow}}}
\newcommand{\ito}{\hookrightarrow}

\newcommand{\PB}[1]{\arrow[#1,phantom,"\scalebox{1.6}{\color{black}$\lrcorner$}",very near start]}
\newcommand{\scaleu}[2][1.2]{{\scalebox{#1}{$#2$}}}
\NewDocumentCommand{\fib}{O{n} O{2.3} mmm}{%
	\begin{cd}*[#2][5]
		{#3}\ifx#1n{\arrow[d,"{\,\scaleu{#4}}"]}\else{\ifx#1i{\arrow[d,hookrightarrow,"{\,\scaleu{#4}}"]}\else{\ifx#1e{\arrow[d,equal,"{\,\scaleu{#4}}"]}\else{\ifx#1R{\arrow[d,Rightarrow,"{\,\scaleu{#4}}"]}\fi}\fi}\fi}\fi\\
		{#5}\ifx#1o{\arrow[u,"{\,\scaleu{#4}}"']}\fi
	\end{cd}\xspace
}
\newcommand{\Ar}[4][]{\arrow[#2,"{#3}"{#1},""{name=#4, anchor=center}]}
\newcommand{\Ars}[4][]{\arrow[#2,"{#3}"'{#1},""{name=#4, anchor=center}]}
\newcommand{\Arb}[6][]{\arrow[#2,"{#3}"{#1},from=#4,to=#5,shorten <= #6 em, shorten >= #6 em]}
\newcommand{\Arbs}[6][]{\arrow[#2,"{#3}"'{#1},from=#4,to=#5,shorten <= #6 em, shorten >= #6 em]}
\NewDocumentCommand{\tc}{s t+ O{7} O{30} O{} O{} O{} o}{
	\def\footc##1##2##3##4##5{%
		\FPmul\Mulresulttwo{#3}{#3}%
		\FPmul\Mulresult{0.0026}{\Mulresulttwo}%
		\IfBooleanTF{#1}{\begin{cd}*}{\begin{cd}}[#3][#3]
				{##1}\Ar[#5]{r,bend left=#4}{##3}{A}\Ars[#6]{r,bend right=#4}{##4}{B}\&{##2}
				\IfBooleanTF{#2}{\Arb[description,pos=0.49]}{\Arb}{Rightarrow #7}{\mkern1mu {##5}}{A}{B}{\IfNoValueTF{#8}{\Mulresult}{#8}}
		\end{cd}}%
		\footc}
\NewDocumentCommand{\tcv}{s t' O{5} O{38} mmmmm}{
		\FPmul\Mulresulttwo{#3}{#3}%
		\FPmul\Mulresult{0.0026}{\Mulresulttwo}%
		\IfBooleanTF{#1}{\begin{cd}*}{\begin{cd}}[#3][#3]
				{#5}\IfBooleanTF{#2}{\Ars{d,leftarrow,bend right=#4}{#7}{A}\Ar{d,leftarrow,bend left=#4}{#8}{B}}{\Ars{d,bend right=#4}{#7}{A}\Ar{d,bend left=#4}{#8}{B}}\\{#6}
				\Arb{Rightarrow}{#9}{A}{B}{\Mulresult}
		\end{cd}}
\NewDocumentCommand{\sq}{s O{n} O{6} O{6} O{} O{2.7} O{2.2} O{0.5} O{n}}{%
			\def\foosq##1##2##3##4##5##6##7##8{%
				\IfBooleanTF{#1}{\begin{cd}*}{\begin{cd}}[#3][#4]
						{##1}\ifx#2p{\PB{rd}}\fi\arrow[r,"{##5}"]\ifx#9l{\arrow[d,equal,"{##6}"']}\else{\arrow[d,"{##6}"']}\fi\&{##2}\ifx#9r{\arrow[d,equal,"{##7}"]}\else{\arrow[d,"{##7}"]}\fi\ifx#2l{\arrow[ld,Rightarrow,shorten <=#6ex,shorten >=#7ex,"{#5}"{pos=#8}]}\fi \ifx#2i{\arrow[ld,twoiso,shorten <=#6ex,shorten >=#7ex,"{#5}"{pos=#8}]}\fi\\
						{##3}\ifx#9d{\arrow[r,equal,"{##8}"']}\else{\arrow[r,"{##8}"']}\fi\ifx#2o{\arrow[ur,Rightarrow,shorten <=#6ex,shorten >=#7ex,"{#5}"{pos=#8}]}\fi\&{##4}
				\end{cd}}%
				\foosq}
\NewDocumentCommand{\tr}{s O{4.5} O{6.5} O{0} O{0} O{n} O{0} O{} O{0}}{%
				\def\footr##1##2##3##4##5##6{%
					\IfBooleanTF{#1}{\begin{cd}*}{\begin{cd}}[#3][#2]
							{##1}\arrow[rr,"{##4}"]
							\Ars[inner sep =0.2ex]{dr}{##5}{A}\&[#4ex]\&[#5ex]{##2}\Ar[inner sep =0.2ex]{ld}{##6}{B}\\
							\&{##3}
							\ifx#6l{\Arb{Rightarrow,shift right=#7em}{#8}{A}{B}{#9}}\else{\ifx#6o{\Arbs{Rightarrow,shift right=#7em}{#8}{B}{A}{#9}}\else{\ifx#6i{\Arbs[inner sep=0.9ex]{iso,shift right=#7em}{#8}{A}{B}{#9}}\else{\ifx#6e{\Arb{equal,shift right=#7em}{#8}{A}{B}{#9}}\else{}\fi}\fi}\fi}\fi
					\end{cd}}%
					\footr}

\renewcommand{\ker}[1]{\operatorname{ker}(#1)}
\newcommand{\coker}[1]{\operatorname{coker}(#1)}

\begin{document}
					
\title{Comonadic approach to pretorsion theories}
					
\author{Elena Caviglia}
\address{(Elena Caviglia) Department of Mathematical Sciences, Stellenbosch University, South Africa. National Institute for Theoretical and Computational Sciences (NITheCS), Stellenbosch, South Africa.}
\email{elena.caviglia@outlook.com}
					
\author{Zurab Janelidze}
\address{(Zurab Janelidze) Department of Mathematical Sciences, Stellenbosch University, South Africa.
						National Institute for Theoretical and Computational Sciences (NITheCS), Stellenbosch, South Africa.}
\email{zurab@sun.ac.za}
					
\author{Luca Mesiti}
\address{(Luca Mesiti) Department of Mathematical Sciences, Stellenbosch University, South Africa.}
\email{luca.mesiti@outlook.com}
					


\begin{abstract} 
    We present a comonadic approach to pretorsion theories on semiexact categories, i.e.\ categories equipped with a closed ideal of null morphisms that admits all kernels and all cokernels. We first prove that bihereditary pretorsion theories are comonadic in a 2-dimensional sense over the 2-category of semiexact categories with naturally chosen 1-cells. We then extend the built pseudo-comonad to guarantee that all pretorsion theories are pseudo-coalgebras. But interestingly, not all pseudo-coalgebras are pretorsion theories. Rather, pseudo-coalgebras give a generalized notion of pretorsion theory.
\end{abstract}
					
\maketitle
					
\setcounter{tocdepth}{1}
\tableofcontents

\section{Introduction} 


Following the framework of \cite{FFG21}, a \emph{pretorsion theory} on a category $\mathcal{C}$ consists of a pair $(\mathcal{T},\mathcal{F})$ of full replete subcategories that satisfy the following axioms. Denote by
\[
\mathcal{Z} = \mathcal{T} \cap \mathcal{F}
\]
the intersection of the two classes, and let $\mathcal{N}$ be the ideal of morphisms in $\mathcal{C}$ that factor through an object of $\mathcal{Z}$. The defining properties of a pretorsion theory are then expressed relative to this ideal. Firstly, the interaction between the two classes is trivial modulo $\mathcal{N}$:
\begin{itemize}
\item[(T1)] Any morphism whose domain lies in $\mathcal{T}$ and whose codomain lies in $\mathcal{F}$ belongs to $\mathcal{N}$; equivalently, every such morphism factors through an object of $\mathcal{Z}$.
\end{itemize}
Secondly, every object of the ambient category admits a canonical decomposition with respect to $(\mathcal{T},\mathcal{F})$:
\begin{itemize}
\item[(T2)] For each object $C \in \mathcal{C}$, there exists a sequence
\[
T \longrightarrow C \longrightarrow F
\]
that is short exact relative to the ideal $\mathcal{N}$, in the sense that $T \to C$ is a kernel of $C \to F$ and, conversely, $C \to F$ is a cokernel of $T \to C$.
\end{itemize}
Here, kernels and cokernels are understood in the relative sense determined by $\mathcal{N}$. These notions are standard and well developed; we refer to \cite{FFG21} and the references therein for details.

Pretorsion theories significantly extend the classical concept of torsion theories in abelian categories, originally introduced in \cite{D66}. Several intermediate generalisations have appeared in the literature (see, for example, \cite{GJ20,FFG21}). One of the main outcomes of \cite{FFG21} is that pretorsion theories strike a balance between generality and structure: despite their minimal axiomatic basis, they allow for the derivation of many nontrivial properties, while at the same time encompassing a wide and diverse range of examples.

In this paper we are concerned with a special type of pretorsion theories: those where the underlying category together with the ideal $\mathcal{N}$ forms a semiexact category in the sense of \cite{G13}. We are thus within the scope of torsion theories studied in \cite{GJ20}, which is still a fairly wide scope, although it misses some of the exotic examples discussed in \cite{FFG21}. We may interchangeably refer to these torsion theories as \emph{pretorsion theories in semiexact categories} or simply \emph{pretorsion theories}.

Despite an extensive literature on pretorsion theories, including various different subtypes of pretorsion theories, abundant examples, properties and characterizations, little is known about the category of pretorsion theories and appropriate functors between them. Our joint work \cite{CavJanMes25Rec} is among the very first papers that investigate structures for a suitable category of pretorsion theories. There, we presented a monadic approach to capture a special type of torsion theories in pointed categories, that is interestingly connected with rectangular bands from semigroup theory. More precisely, we proved that the category of rectangular torsion theories, i.e.\ those torsion theories $(\C, \T,\F)$ for which the canonical functor $\C \to \T \x \F$ is an equivalence of categories, is equivalent to the category of pseudo-algebras for the squaring 2-monad on the category of pointed categories. This partial monadicity result led us to the topic of the present paper: a comonadic approach to the study of a much wider class of pretorsion theories.

In this paper, we prove two different comonadicity results for pretorsion theories. The first one (\thex \ref{theorbih}) exhibits a pseudo-comonad on the 2-category of semiexact categories and functors preserving kernels and cokernels, whose pseudo-coalgebras are precisely bihereditary pretorsion theories. The second one (\thex \ref{theoromegaex}) extends this pseudo-comonad to one on the 2-category of semiexact categories and functors preserving exact sequences. We prove that all pretorsion theories on semiexact categories are pseudo-coalgebras for the extended pseudo-comonad. But interestingly, there are pseudocoalgebras that are not pretorsion theories in the usual sense. Indeed, the pseudo-coalgebras give a generalization of pretorsion theories that we plan to investigate further in future work.

\begin{notation}
    Whenever we say \emph{pseudo-coalgebra} for a 2-monad or a pseudo-comonad, we mean a normal pseudo-coalgebra.

    Whenever $(\C,\T,\F)$ is a pretorsion theory, we fix a choice of short exact sequences associated to every object of $\C$. 
\end{notation}

\section{The category of short exact sequences}

The content and many of the results of this section are immediate from \cite{G13}, and many aspects of such results follow from more general results (some of which are contained in \cite{G13}). We redevelop them in detail, though, for the sake of completeness.  

We study the category $\ses{\C}$ of short exact sequences in a category $\C$ equipped with a closed ideal of null-morphisms $\N$. We present a useful characterization of kernels and cokernels, and provide an explicit construction. We use this to completely characterize and study short exact sequences in $\Ses{\C}$.

Finally, we exhibit a pretorsion theory on $\ses{\C}$, that will later correspond to the cofree bihereditary pretorsion theory on $\C$.

The content of this section will give the fundamental building blocks for the construction of the pseudo-comonads of the following sections.

Throughout the paper $\C$ will be a category equipped with a closed ideal $\N$ of null morphisms that has all $\N$-kernels and $\N$-cokernels. The category of short exact sequences in $\C$ is defined as follows.

\begin{definition}
We define $\ses{\C}$ to be the category given by the following:
\begin{description}
    \item[an object of $\ses{\C}$ is] a short exact sequence in $\C$;
    \item[a morphism in $\ses{\C}$ from $X \ar{f} Y  \ar{g} Z$ to $X' \ar{f'} Y'  \ar{g'} Z'$ is] a triple $(u,v,w)$ of morphisms in $\C$ making the following diagram commute
    \begin{cd}[6]
	X \& Y \& Z \\
X' \& Y' \& Z'
	\arrow["f", from=1-1, to=1-2]
\arrow["g", from=1-2, to=1-3]
	\arrow["u"', from=1-1, to=2-1]
	\arrow["v"', from=1-2, to=2-2]
	\arrow["w", from=1-3, to=2-3]
	\arrow["f'"', from=2-1, to=2-2]
	\arrow["g'"', from=2-2, to=2-3]
\end{cd}
\end{description}
Composition and identities are then defined using the ones of $\C$ on every component of the triples. 
\end{definition}

\begin{remark}
    In all the diagrams involving morphisms of short exact sequences throughout the paper, the objects of $\ses{\C}$ will be written orizontally while the morphisms will be triples of vertical arrows.
\end{remark}

We can equip the category $\ses{\C}$ with a closed ideal of null morphisms constructed from the closed ideal $\N$. 

\begin{definition}\label{def[N]}
    The closed ideal $[\N]$ is the ideal of $\ses{\C}$ generated by the short exact sequences of the form 
    $$X \aiso{f} Y  \aiso{g} Z$$
    where both morphisms are isomorphisms in $\C$.
    \end{definition}

\begin{remark}\label{remcharnull}
It is straightforward to see that given a short exact sequence $X \aiso{f} Y  \aiso{g} Z$ the objects $X, Y$ and $Z$ are null objects in $\C$. Indeed, all $\N$-kernels and $\N$-cokernels of isomorphisms are null objects. Conversely, a short exact sequence involving three null objects must of the form $\bullet \aiso{} \bullet  \aiso{} \bullet$.
\end{remark}

We prove a characterization of kernels, cokernels and short exact sequences in $\ses{\C}$. Before that, we recall the following well-known facts.

\begin{lemma}\label{lemmino}
    Let $\C$ be a category with a closed ideal of null morphisms $\N$. The following fact hold:
    \begin{itemize} 
    \item[(i)] the kernel of any isomorphism is a null object, and the dual holds for cokernels of isomorphisms; 
    \item[(ii)] the kernel of a null morphism is the identity, and, dually, the cokernel of a null morphism is the identity.
    \end{itemize}
We thus have the following canonical short exact sequences associated to any object $Y \in \C$:
$$K(\id{Y}) \ar{\ker{\id{Y}}} Y \aeqq{} Y \qquad \text{and} \qquad Y\aeqq{} Y\ar{\coker{\id{Y}}} C(\id{Y})$$
\end{lemma}

\begin{prop}\label{propcharkerses}
Let $(u,v,w)$ and $(u',v',w')$ be composable morphisms in $\ses{\C}$ as in the following diagram
\begin{eqD*}
{\begin{cd}*[6]
	X \& Y \& Z \\
	{X'} \& {Y'} \& {Z'} \\
	{X''} \& {Y''} \& {Z''}
	\arrow["f", from=1-1, to=1-2]
	\arrow["u"', from=1-1, to=2-1]
	\arrow["g", from=1-2, to=1-3]
	\arrow["v"', from=1-2, to=2-2]
	\arrow["w", from=1-3, to=2-3]
	\arrow["{f'}"', from=2-1, to=2-2]
	\arrow["{u'}"', from=2-1, to=3-1]
	\arrow["{g'}"', from=2-2, to=2-3]
	\arrow["{v'}"', from=2-2, to=3-2]
	\arrow["{w'}", from=2-3, to=3-3]
	\arrow["{f''}"', from=3-1, to=3-2]
	\arrow["{g''}"', from=3-2, to=3-3]
\end{cd}}
\end{eqD*}
The morphism $(u,v,w)$ is the $[\N]$-kernel of $(u',v',w')$ if and only if $u$ is the $\N$-kernel of $u'$ and $v$ is the $[\N]$-kernel of $v'$. Dually, the morphism $(u',v',w')$ is the 
$[\N]$-cokernel of $(u,v,w)$ if and only if $v$ is the $\N$-cokernel of $v$ and $w'$ is the $[\N]$-kernel of $w$.

So, the morphisms $(u,v,w)$ and $(u',v',w')$ form a short exact sequence in $\ses{\C}$ precisely when $Y \ar{v} Y' \ar{v'} Y''$ is a short exact sequence in $\C$, $u=\ker{u'}$ and $w'=\coker{w}$.
\end{prop}

\begin{proof}
    We start by proving that if $(u,v,w)=\ker{(u',v',w')}$ then $u=\ker{u'}$ and $v=\ker{v'}$. 

    Let $M \ar{l} Y$ be a morphism such that the composite $M \ar{l} Y \ar{v'}$ is a null morphism. We need to show that there exists a unique morphism $M \ar{d} Y$ such that $v\c d=l$.
To induce such a morphism $d$, we consider the short exact sequence
$$K(\id{M}) \ar{\ker{\id{M}}} M \aeqq{} M$$
and we apply the universal property of $(u,v,w)$ as $[\N]$-kernel of $(u',v',w')$ as in the following diagram 
\begin{cd}[5][6]
	{K(\id{M})} \&\& M \&\& M \\
	\& X \&\& Y \&\& Z \\
	\& {X'} \&\& {Y'} \&\& {Z'} \\
	\& {X''} \&\& {Y''} \&\& {Z''}
	\arrow["{\ker{\id{M}}}", from=1-1, to=1-3]
	\arrow["{\exists ! c}", dashed, from=1-1, to=2-2]
	\arrow["t"{description}, bend right=15, from=1-1, to=3-2]
	\arrow[equals, from=1-3, to=1-5]
	\arrow["{\exists ! d}", dashed, from=1-3, to=2-4]
	\arrow["l"{description, pos=0.7}, bend right=15, from=1-3, to=3-4]
	\arrow["{\exists ! e}", dashed, from=1-5, to=2-6]
	\arrow["{g' \circ l}"{description, pos=0.7}, bend right=15, from=1-5, to=3-6]
	\arrow["f", from=2-2, to=2-4]
	\arrow["u"', from=2-2, to=3-2]
	\arrow["g", from=2-4, to=2-6]
	\arrow["v", from=2-4, to=3-4]
	\arrow["w", from=2-6, to=3-6]
	\arrow["{f'}"', from=3-2, to=3-4]
	\arrow["{u'}"', from=3-2, to=4-2]
	\arrow["{g'}"', from=3-4, to=3-6]
	\arrow["{v'}", from=3-4, to=4-4]
	\arrow["{w'}", from=3-6, to=4-6]
	\arrow["{f''}"', from=4-2, to=4-4]
	\arrow["{g''}"', from=4-4, to=4-6]
\end{cd}
Here the morphism $t$ is the unique morphism induced by the universal property of $f'$ as kernel of $g'$ applied to the null morphism $g'\circ \ker{\id{M}}$ and thus the triple $(t,l,g'\circ l)$ is a morphism in $\ses{\C}$ by construction. The induced morphism $M \ar{d} Y$ is such that $v\c d=l$. Moreover, its uniqueness follows by the uniqueness of the morphism $(c,d,e)$ in $\ses{C}$. Indeed, given a morphism $d'$ such that $v\circ d'=l$, the triple $(c',d',g\c d')$, where $c'$ is induced by the universal property of $f'$ as kernel of $g'$ for the null morphism $g\c d'$, is another morphism in $\ses{\C}$ that fits into the previous diagram. So, by uniqueness of $(c,d,e)$ it must be $d=d'$. This shows that $v=\ker{v'}$. 

To prove that $u=\ker{u'}$, let us now consider a morphism $L \ar{q} X'$ such that $u'\circ q$ is a null morphism. We need to show that there exists a unique morphism $L \ar{a} Y$ such that $u\c a=q$. To induce such a morphism, we consider the short exact sequence
$$L\aeqq{} L \ar{\coker{\id{L}}} C(\id{L})$$
and we apply the universal property of $(u,v,w)$ as $[\N]$-kernel of $(u',v',w')$ as in the following diagram
\begin{cd}[5][6]
    L \&\& L \&\& {C(\id{L})} \\
	\& X \&\& Y \&\& Z \\
	\& {X'} \&\& {Y'} \&\& {Z'} \\
	\& {X''} \&\& {Y''} \&\& {Z''}
	\arrow[equals, from=1-1, to=1-3]
	\arrow["{\exists ! a}", dashed, from=1-1, to=2-2]
	\arrow["q"{description}, bend right=15, from=1-1, to=3-2]
	\arrow["{\coker{\id{L}}}", from=1-3, to=1-5]
	\arrow["{\exists ! b}", dashed, from=1-3, to=2-4]
	\arrow["{f' \circ q}"{description, pos=0.7}, bend right=15, from=1-3, to=3-4]
	\arrow["{\exists ! c}", dashed, from=1-5, to=2-6]
	\arrow["r"{description, pos=0.7}, bend right=15, from=1-5, to=3-6]
	\arrow["f", from=2-2, to=2-4]
	\arrow["u"', from=2-2, to=3-2]
	\arrow["g", from=2-4, to=2-6]
	\arrow["v", from=2-4, to=3-4]
	\arrow["w", from=2-6, to=3-6]
	\arrow["{f'}"', from=3-2, to=3-4]
	\arrow["{u'}"', from=3-2, to=4-2]
	\arrow["{g'}"', from=3-4, to=3-6]
	\arrow["{v'}", from=3-4, to=4-4]
	\arrow["{w'}", from=3-6, to=4-6]
	\arrow["{f''}"', from=4-2, to=4-4]
	\arrow["{g''}"', from=4-4, to=4-6]
\end{cd}
Here the morphism $r$ is the unique morphism induced by the universal property of $g'$ as cokernel of $f'$ applied to the null morphism $r \circ \coker{\id{M}}$ and thus the triple $(q,f'\circ q,r)$ is a morphism in $\ses{\C}$ by construction. The induced morphism $L \ar{a} X$ is such that $u\c a=q$ by construction. Moreover, its uniqueness follows by the uniqueness of the morphism $(a,b,c)$ in $\ses{C}$ analogously to what we showed in the proof that $v=\ker{v'}$. This conclude the proof that $u=\ker{u'}$. 

We now prove that if $u=\ker{u'}$ and $v=\ker{v'}$ in $\C$ then $(u,v,w)=\ker{(u',v',w')}$ in $\ses{\C}$. Let $L \ar{l} M \ar{r} R$ be a short exact sequence in $\C$ and let $(q,s,t)$ be a morphism in $\ses{\C}$ as in the diagram below such that the composite $(u',v',w') \c (q,s,t)$ is in $[\N]$. We need to prove that there exist a unique morphism $(a,b,c)$ is $\ses{\C}$ that fits in the following diagram
\begin{cd}[5][6]
    L \&\& M \&\& R \\
	\& X \&\& Y \&\& Z \\
	\& {X'} \&\& {Y'} \&\& {Z'} \\
	\& {X''} \&\& {Y''} \&\& {Z''}
	\arrow["l", from=1-1, to=1-3]
	\arrow["{\exists ! a}", dashed, from=1-1, to=2-2]
	\arrow["q"{description}, bend right=15, from=1-1, to=3-2]
	\arrow["r", from=1-3, to=1-5]
	\arrow["{\exists ! b}", dashed, from=1-3, to=2-4]
	\arrow["s"{description, pos=0.7}, bend right=15, from=1-3, to=3-4]
	\arrow["{\exists ! c}", dashed, from=1-5, to=2-6]
	\arrow["t"{description, pos=0.7}, bend right=15, from=1-5, to=3-6]
	\arrow["f", from=2-2, to=2-4]
	\arrow["u"', from=2-2, to=3-2]
	\arrow["g", from=2-4, to=2-6]
	\arrow["v", from=2-4, to=3-4]
	\arrow["w", from=2-6, to=3-6]
	\arrow["{f'}"', from=3-2, to=3-4]
	\arrow["{u'}"', from=3-2, to=4-2]
	\arrow["{g'}"', from=3-4, to=3-6]
	\arrow["{v'}", from=3-4, to=4-4]
	\arrow["{w'}", from=3-6, to=4-6]
	\arrow["{f''}"', from=4-2, to=4-4]
	\arrow["{g''}"', from=4-4, to=4-6]
\end{cd}
Since $u=\ker{u'}$, we can induce $a$ by applying the universal property of the kernel. Analogously, we can induce $b$ via the universal property of $v=\ker{v'}$. To induce $c$, we then use the fact that $r=\coker{l}$ and that the morphism $g\c b\c l=g \c f \c a$ is null. And the uniqueness of the triple $(a,b,c)$ directly follows from the uniqueness of the three morphisms. So we conclude that $(u,v,w)=\ker{(u',v',w')}$. 

The proof of the dual statement for cokernels can be easily obtained by dualizing this proof. 

Finally, the characterization of short exact sequences immediately follows by putting together the characterizations of kernels and cokernels.
\end{proof}

We now give an explicit construction for kernels and cokernels in $\ses{\C}$.

\begin{prop}\label{kersesconstr}
Let $(u',v',w')$ be a morphism in $\ses{\C}$ as in the diagram below. Then the kernel of $(u',v',w')$ can be constructed as in the following diagram   
\begin{cd}
    {K(u')} \& {K(v')} \& {C(f)} \\
	{X'} \& {Y'} \& {Z'} \\
	{X''} \& {Y''} \& {Z''}
	\arrow["{\exists ! f}", dashed, from=1-1, to=1-2]
	\arrow["{\ker{u'}}"', from=1-1, to=2-1]
	\arrow["{\coker{f}}", from=1-2, to=1-3]
	\arrow["{\ker{v'}}", from=1-2, to=2-2]
	\arrow["{\exists ! w}", dashed, from=1-3, to=2-3]
	\arrow["{f'}"', from=2-1, to=2-2]
	\arrow["{u'}"', from=2-1, to=3-1]
	\arrow["{g'}"', from=2-2, to=2-3]
	\arrow["{v'}", from=2-2, to=3-2]
	\arrow["{w'}", from=2-3, to=3-3]
	\arrow["{f''}"', from=3-1, to=3-2]
	\arrow["{g''}"', from=3-2, to=3-3]
\end{cd}
    where $f$ is the unique morphism induced by the universal property of $\ker{v'}$ applied to the null morphism $f'' \c u' \c \ker{u'}$ and $w$ is the unique morphism induced by the universal property of $\coker{f}$ applied to the null morphism $g' \c \ker{v'} \c f=g' \c f' \c \ker{u'}$.

    And the dual statement holds for cokernels.
\end{prop}

\begin{proof}
    We first prove that 
    $$ K(u') \ar{f} K(v') \ar{\coker{f}} C(f)$$
    is a short exact sequence in $\C$. To do so, we need to show that $f$ is the kernel of $\coker{f}$. Let $M \ar{m} K(v')$ be a morphism in $\C$ such that $\coker{f} \c m$ is a null morphism. We need to show that there exists a unique morphism $L \ar{l} K(u')$ such that $f \c l= m$.
    We induce such a morphism $l$ as in the following diagram 
    \begin{cd}[4.5][6]
        M \\[-3ex]
	\& {K(u')} \&\& {K(v')} \&\& {C(f)} \\
	\& {X'} \&\& {Y'} \&\& {Z'} \\
	\& {X''} \&\& {Y''} \&\& {Z''}
	\arrow["{\exists ! l}", dashed, from=1-1, to=2-2]
	\arrow["m", bend left=15, from=1-1, to=2-4]
	\arrow["{\exists ! s}"', bend right=15, dashed, from=1-1, to=3-2]
	\arrow["{f}", from=2-2, to=2-4]
	\arrow["{\ker{u'}}", from=2-2, to=3-2]
	\arrow["{\coker{f}}", from=2-4, to=2-6]
	\arrow["{\ker{v'}}", from=2-4, to=3-4]
	\arrow["{w}", from=2-6, to=3-6]
	\arrow["{f'}"', from=3-2, to=3-4]
	\arrow["{u'}"', from=3-2, to=4-2]
	\arrow["{g'}"', from=3-4, to=3-6]
	\arrow["{v'}", from=3-4, to=4-4]
	\arrow["{w'}", from=3-6, to=4-6]
	\arrow["{f''}"', from=4-2, to=4-4]
	\arrow["{g''}"', from=4-4, to=4-6]
    \end{cd}
    Here $s$ is the unique morphism induced by the universal property of $f'=\ker{g'}$ applied to the null morphism $g' \c \ker{v'} \c m=w \c \coker{f} \c m$. And $l$ is the unique morphism induced by the universal property of $\ker{u'}$ applied to the null morphism $u' \c s$. Notice that $u' \c s$ is null since $f'' \c u' \c s=v' \c \ker{v'} \c m$ is null and $f''$ reflects null morphisms because it is a kernel. The fact that $f \c l= m$ then follows from the chain of equalities $$\ker{v'} \c m= f' \c s= f' \c \ker{u'} \c l= \ker{v'} \c f \c l$$
    thanks to the fact that $\ker{v'}$ is a monomorphism. Moreover, the uniqueness of $l$ as morphism such that $f \c l= m$ follows by the uniqueness of $s$ and the uniqueness of $l$ as morphism such that $\ker{u'} \c l=s$. This conclude the proof that $f= \ker{\coker{f}}$ and so $ K(u') \ar{f} K(v') \ar{\coker{f}} C(f)$
    is a short exact sequence.
By \prox\ref{propcharkerses}, we conclude that $(\ker{u'}, \ker{v'}, w)$ is the kernel of $(u',v',w')$. And the dual statement for cokernels can be proved analogously.
\end{proof}

We have thus proved the following result.

\begin{prop}\label{propsesN}
 Let $\C$ be a category equipped with a closed ideal $\N$ of null morphisms that has all $\N$-kernels and $\N$-cokernels. Then the category $\ses{\C}$ equipped with the closed ideal $[\N]$ of Definition \ref{def[N]} has all $[\N]$-kernels and $[\N]$-cokernels.
\end{prop}

We now define a pretorsion theory on the category $\ses{\C}$.

\begin{prop}\label{proppretorsses}
    Let $\T$ be the class of short exact sequences in $\C$ of the form $\bullet \aiso{} \bullet \to\bullet$ and let $\F$ be the class of short exact sequences in $\C$ of the form $\bullet \to \bullet \aiso{}\bullet$. Then $(\C,\T,\F)$ is a pretorsion theory on $\ses{\C}$. 
\end{prop}

\begin{proof}
    Let $(u,v,w)$ be a morphism in $\ses{\C}$ as in the following diagram
    \begin{cd}
        X \& Y \& Z \\
	{X'} \& {Y'} \& {Z'}
	\arrow["f"', from=1-1, to=1-2,aiso]
	\arrow["u"', from=1-1, to=2-1]
	\arrow["g"', from=1-2, to=1-3]
	\arrow["v", from=1-2, to=2-2]
	\arrow["w", from=1-3, to=2-3]
	\arrow["{f'}"', from=2-1, to=2-2]
	\arrow["{g'}"', from=2-2, to=2-3,aiso]
    \end{cd}
    from a short exact sequence in $\T$ to one in $\F$. Then $(u,v,w)$ factors through the null short exact sequence given by the identity of $X'$. Indeed, the following diagram commutes 
    \begin{cd}
      X \& Y \& Z \\
	{X'} \& {X'} \& {X'} \\
	{X'} \& {Y'} \& {Z'}
	\arrow["f"', from=1-1, to=1-2,aiso]
	\arrow["u"', from=1-1, to=2-1]
	\arrow["u"{description, pos=0.6}, bend right=35, from=1-1, to=3-1]
	\arrow["g"', from=1-2, to=1-3]
	\arrow["{u \circ f^{-1}}", from=1-2, to=2-2]
	\arrow["v"{description, pos=0.7}, bend right=35, from=1-2, to=3-2]
	\arrow["{\exists ! a}", dashed, from=1-3, to=2-3]
	\arrow["w"{description, pos=0.7}, bend right=35, from=1-3, to=3-3]
	\arrow[equals, from=2-1, to=2-2]
	\arrow[equals, from=2-1, to=3-1]
	\arrow[equals, from=2-2, to=2-3]
	\arrow["{f'}", from=2-2, to=3-2]
	\arrow["{g'\circ f'}", from=2-3, to=3-3]
	\arrow["{f'}"', from=3-1, to=3-2]
	\arrow["{g'}"', from=3-2, to=3-3,aiso]
    \end{cd}
    Here $a$ is the unique morphism induced by the universal property of $g=\coker{f}$ applied to the null morphism $u$. Notice that $w= g' \c f' \c a$ because they are equal when precomposed by $g$, which is a cokernel and thus it is an epimorphism. So we have proved that the morphism $(u,v,w)$ is in $[\N]$. 

    Let now $X \ar{f} Y \ar{g} Z$ be an object of $\ses{\C}$. We can associate to it the short exact sequence
    \begin{cd}[5][7]
        X \& X \& {C(\id{X})} \\
	X \& Y \& Z \\
	{K(\id{Z})} \& Z \& Z
	\arrow[equals, from=1-1, to=1-2]
	\arrow[equals, from=1-1, to=2-1]
	\arrow["{\coker{\id{X}}}", from=1-2, to=1-3]
	\arrow["f", from=1-2, to=2-2]
	\arrow["{\exists ! w}", dashed, from=1-3, to=2-3]
	\arrow["f", from=2-1, to=2-2]
	\arrow["{\exists ! u}"', dashed, from=2-1, to=3-1]
	\arrow["g", from=2-2, to=2-3]
	\arrow["g", from=2-2, to=3-2]
	\arrow[equals, from=2-3, to=3-3]
	\arrow["{\ker{\id{Z}}}"', from=3-1, to=3-2]
	\arrow[equals, from=3-2, to=3-3]
    \end{cd}
    where the morphism $u$ is induced by the universal property of $\ker{\id{Z}}$ applied to the null morphism $g \c f$ and the morphism $w$ is induced by the universal property of $\coker{\id{X}}$ applied to the null morphism $g \c f$. Since $u$ is a null morphism (as it factors through the null object $K(\id{Z})$, we have that $\ker{u}=\id{X}$. 
    The sequence in $\Ses{C}$ drawn above is a short exact sequence in $\Ses{C}$, thanks to the explicit characterization of \prox\ref{propcharkerses}. Indeed, the middle column is a short exact sequence (it actually coincides with the starting one). And, since $K(\id{Z})$ and $C(\id{X})$ are null objects, the left column is a kernel and the right column is a cokernel, thanks to \lemx\ref{lemmino}.
\end{proof}

\section{A comonad for bihereditary pretorsion theories}

In this section, we make the first step towards capturing pretorsion theories via a comonadic approach. We construct a pseudo-comonad whose pseudo-coalgebras are bihereditary pretorsion theories. In the following section, we will then extend this pseudo-comonad to guarantee that all pretorsion theories are pseudo-coalgebras.

We start by recalling what a (co)hereditary pretorsion theory is.

\begin{definition}
    A pretorsion theory $(\C,\T,\F)$ is called \dfn{hereditary} if the functor $\C \ar{T} \T \ito \C$ preserves kernels, and is called \dfn{cohereditary} if it has the dual property that $\C \ar{F} \F \ito \C$ preserves cokernels. A pretorsion theory is called bihereditary if it is both hereditary and cohereditary.
\end{definition}

\begin{example}
    The classical torsion theory $(\Ab, \{\text{torsion abelian groups}\}, \{\text{torsion-free abelian groups}\})$ is hereditary but not cohereditary.
\end{example}

\begin{example}
    It is straightforward to prove that all rectangular torsion theories, defined in our joint work \cite{CavJanMes25Rec}, are bihereditary. Interestingly, as we proved in \cite{CavJanMes25Rec}, rectangular torsion theories can be captured as pseudo-algebras for a 2-monad on the 2-category of pointed categories and functors that preserve the zero object. In this section we will prove that they are also pseudo-coalgebras for a pseudo-comonad.
\end{example}

\begin{example}
    As shown in \cite{FFG21},
    given a  finite preordered set $(X=\{1,2,\ldots,n\}, \leq)$ the triple $(X,T,F)$ with $T,F\cont X$ is a pretorsion theory iff 
\begin{itemize}
\item[(i)] $T \cup F=X$;
\item[(ii)]  $1\in T$ and $n\in F$;
\item[(ii)]  for every $i=1,\ldots, n-1$ if $i\in T$ and $i+1\in F$, then either $i\in F$ or $i+1\in T$.
\end{itemize}
It is easy to prove that all such pretorsion theories are bihereditary.
\end{example}

\begin{proposition}
    The pretorsion theory on $\Ses{\C}$ that we built in \prox\ref{proppretorsses} is bihereditary.
\end{proposition}
\begin{proof}
    Let $(u,v,w)$ be the kernel of $(u',v',w')$ as in the following diagram.
    \begin{eqD*}
{\begin{cd}*[3][3]
	X \& Y \& Z \\
	{X'} \& {Y'} \& {Z'} \\
	{X''} \& {Y''} \& {Z''}
	\arrow["f", from=1-1, to=1-2]
	\arrow["u"', from=1-1, to=2-1]
	\arrow["g", from=1-2, to=1-3]
	\arrow["v"', from=1-2, to=2-2]
	\arrow["w", from=1-3, to=2-3]
	\arrow["{f'}"', from=2-1, to=2-2]
	\arrow["{u'}"', from=2-1, to=3-1]
	\arrow["{g'}"', from=2-2, to=2-3]
	\arrow["{v'}"', from=2-2, to=3-2]
	\arrow["{w'}", from=2-3, to=3-3]
	\arrow["{f''}"', from=3-1, to=3-2]
	\arrow["{g''}"', from=3-2, to=3-3]
\end{cd}}
\end{eqD*}
The functor $T\: \ses{\C} \to \T_{\ses{\C}}$ applied to this kernel produces
\begin{cd}
    X \& X \& {C(\id{X})} \\
	{X'} \& {X'} \& {C(\id{X'})} \\
	{X''} \& {X''} \& {C(\id{X''})}
	\arrow[equals, from=1-1, to=1-2]
	\arrow["u"', from=1-1, to=2-1]
	\arrow["{\coker{\id{X}}}", from=1-2, to=1-3]
	\arrow["u", from=1-2, to=2-2]
	\arrow["{\exists ! a}", from=1-3, to=2-3]
	\arrow[equals, from=2-1, to=2-2]
	\arrow["{u'}"', from=2-1, to=3-1]
	\arrow["{\coker{\id{X'}}}"', from=2-2, to=2-3]
	\arrow["{u'}", from=2-2, to=3-2]
	\arrow["{\exists !a'}", from=2-3, to=3-3]
	\arrow[equals, from=3-1, to=3-2]
	\arrow["{\coker{\id{X''}}}"', from=3-2, to=3-3]
\end{cd}
where $a$ and $a'$ are induced by the universal property of the involved cokernels. We need to show that this diagram exhibits a cokernel in $\ses{\C}$. But this is true by \prox \ref{propcharkerses} since $u=\ker{u'}$. This proves that the pretorsion theory is hereditary. The proof that it is also cohereditary is dual.
\end{proof}

Bihereditary pretorsion theories thus give a non-trivial interesting class of pretorsion theories.

We now construct the ground 2-category of our pseudo-comonad for bihereditary pretorsion theories. 

We define $\ClIdl$ to be the 2-category given by the following:
\begin{description}
    \item[an object of $\ClIdl$ is] a pair $(\C,\N)$ of a category $\C$ equipped with a closed ideal $\N$, that has all $\N$-kernels and all $\N$-cokernels;
    \item[a morphism from $(\C,\N)$ to $(\D,\M)$ is] a functor $F\:\C\to \D$ that preserves all $\N$-kernels and $\N$-cokernels;
    \item[a 2-cell is] just a natural transformation between functors.
\end{description}

\begin{proposition}\label{propclidl}
    $\ClIdl$ is indeed a 2-category.

    Moreover, every morphism in $\ClIdl$ automatically preserves null objects and short exact sequences.
\end{proposition}
\begin{proof}
    The proof is straightforward.
\end{proof}

Now that we have defined the ground 2-category $\ClIdl$, we want to present the 2-category of pretorsion theories that will capture via our comonadic approach in this section.

In the literature, there is no standard 2-category of pretorsion theories. And it is not clear a priori what the morphisms between categories that are equipped with a pretorsion theory should be. One of the outcomes of this paper is to shed light on the choice of such morphisms. The following 2-category of bihereditary pretorsion theories is what will correspond to the pseudo-coalgebras of our pseudo-comonad.

We define $\BiHPTors$ to be the 2-category given by the following:
\begin{description}
    \item[an object of $\BiHPTors$ is] a bihereditary pretorsion theory $(\C,\T,\F)$ on a category $\C$ that has all kernels and cokernels with respect to the closed ideal $\T\cap \F$;
    \item[a morphism from $(\C,\T,\F)$ to $(\D,\T',\F')$ is] a functor $F\:\C\to \D$ that preserves kernels, cokernels, torsion objects and torsion-free objects:
    \item[a 2-cell is] just a natural transformation between functors.
\end{description}

\begin{proposition}
    $\BiHPTors$ is indeed a 2-category. Moreover, there is a forgetful 2-functor $U\: \BiHPTors\to \ClIdl$.
\end{proposition}
\begin{proof}
    It is straightforward to prove that $\BiHPTors$ is indeed a 2-category. Of course, every pretorsion theory $(\C,\T,\F)$ has the underlying category with a closed ideal $(\C,\T\cap \F)$. By construction, morphisms in $\BiHPTors$ preserve (relative) kernels and cokernels.
\end{proof}

We are now ready to present our pseudo-comonad for bihereditary pretorsion theories.

\begin{construction}\label{conscomonadbih}
    The assignment $\C\mto \Ses{\C}$ that sends every category $\C$ to the category of short exact sequences in $\C$ can be extended to a 2-functor
    \begin{fun}
	\Omega & \: & \ClIdl & \too & \ClIdl \\[0.3ex]
    && \tcv*{\C}{\D}{G}{H}{\alpha} & \mto & \tcv*{\Ses{\C}}{\Ses{\D}}{G}{H}{[\alpha]}
\end{fun}
Indeed, we proved in \ref{propsesN} that if $\C$ is a category equipped with a closed ideal that has all (relative) kernels and cokernels, also $\Ses{\C}$ is such. Given a functor $G\:\C\to \D$ that preserves kernels and cokernels, we define the functor $G\:\Ses{\C}\to \Ses{\D}$ to send a short exact sequence $X\ar{f}Y\ar{g} Z$ to the short exact sequence $G(X)\ar{G(f)}G(Y)\ar{G(g)} G(Z)$ (thanks to \prox\ref{propclidl}). Similarly, the action of $G\:\Ses{\C}\to \Ses{\D}$ on morphisms is given by applying $G\:\C\to \D$ to the whole diagram. It is straightforward to prove that we obtain a functor $G\:\Ses{\C}\to \Ses{\D}$. Moreover, this functor preserves kernels and cokernels, thanks to the explicit characterization of kernels and cokernels in $\Ses{\C}$ that we have proved in \prox\ref{propcharkerses}.

Given a 2-cell $\alpha\:G\aR{}H\:\C\to \D$ in $\ClIdl$, i.e.\ a natural transformation, we construct the natural transformation $[\alpha]\:G\aR{}H\:\Ses{\C}\to \Ses{\D}$ by setting its component on a short exact sequence $X\ar{f}Y\ar{g} Z$ to be the morphism
\begin{cd}[5]
    {G(X)} \& {G(Y)} \& {G(Z)} \\
	{H(X)} \& {H(Y)} \& {H(Z)}
	\arrow["{G(f)}", from=1-1, to=1-2]
	\arrow["{\alpha_X}"', from=1-1, to=2-1]
	\arrow["{G(g)}", from=1-2, to=1-3]
	\arrow["{\alpha_Y}", from=1-2, to=2-2]
	\arrow["{\alpha_Z}", from=1-3, to=2-3]
	\arrow["{H(f)}"', from=2-1, to=2-2]
	\arrow["{H(g)}"', from=2-2, to=2-3]
\end{cd}
The naturality of $\alpha$ guarantees that $[\alpha]$ is natural as well. 

Finally, it is straightforward to prove that $\Omega$ is indeed a 2-functor.
\end{construction}

\begin{theorem}\label{theorbih}
    The 2-functor
    \begin{fun}
	\Omega & \: & \ClIdl & \too & \ClIdl \\[0.3ex]
    && \tcv*{\C}{\D}{G}{H}{\alpha} & \mto & \tcv*{\Ses{\C}}{\Ses{\D}}{G}{H}{[\alpha]}
    \end{fun}
of \conx\ref{conscomonadbih} extends to a pseudo-comonad $\Omega$ on the 2-category $\ClIdl$ of categories equipped with a closed ideal that have all (relative) kernels and cokernels.

    Moreover, the 2-category of pseudo-coalgebras for $\Omega$ is 2-isomorphic to the 2-category $\BiHPTors$ of bihereditary pretorsion theories, over $\ClIdl$. In other words, bihereditary pretorsion theories are comonadic in a 2-dimensional sense over $\ClIdl$.
\end{theorem}
\begin{proof}
   We define the counit $\epsilon\:\Omega\aR{} \Id{}$ to have general component on $\C\in \ClIdl$ given by the functor
    \begin{fun}
	\epsilon_{\C} & \: & \Ses{\C} \hphantom{C}& \too & \C \\[0.3ex]
    && (X\ar{f} Y \ar{g} Z) & \mto & Y
    \end{fun}
    that takes the middle object of short exact sequences, and acts on morphisms in a similar way. The functor $\epsilon_{\C}$ preserves kernels and cokernels, and is thus a morphism in $\ClIdl$, thanks to the explicit characterization of kernels and cokernels in $\Ses{\C}$ that we proved in \prox\ref{propcharkerses}. Moreover, it is easy to show that $\epsilon$ is a 2-natural transformation.

    We then define the comultiplication $\delta\:\Omega\aR{}\Omega\c \Omega$ to have general component on $\C\in \ClIdl$ given by the functor $$\delta_{\C}\:\Ses{\C}\to \Ses{\Ses{\C}}$$
    that sends a short exact sequence $X\ar{f} Y \ar{g} Z$ in $\C$ to the chosen short exact sequence associated to it in the pretorsion theory on $\Ses{C}$ that we built in \ref{proppretorsses}:
    \begin{cd}[5][7]
        X \& X \& {C(\id{X})} \\
	X \& Y \& Z \\
	{K(\id{Z})} \& Z \& Z
	\arrow[equals, from=1-1, to=1-2]
	\arrow[equals, from=1-1, to=2-1]
	\arrow["{\coker{\id{X}}}", from=1-2, to=1-3]
	\arrow["f", from=1-2, to=2-2]
	\arrow["{\exists ! w}", dashed, from=1-3, to=2-3]
	\arrow["f", from=2-1, to=2-2]
	\arrow["{\exists ! u}"', dashed, from=2-1, to=3-1]
	\arrow["g", from=2-2, to=2-3]
	\arrow["g", from=2-2, to=3-2]
	\arrow[equals, from=2-3, to=3-3]
	\arrow["{\ker{\id{Z}}}"', from=3-1, to=3-2]
	\arrow[equals, from=3-2, to=3-3]
    \end{cd}
    It is then easy to extend $\delta_\C$ to a functor $\Ses{\C}\to \Ses{\Ses{\C}}$. Indeed, given a morphism $(u,v,w)$ in $\Ses{\C}$, we use the components of that morphism to define the needed maps on the central row and the central column of the $3\times 3$ grid. In positions $(1,1)$ and $(3,3)$, we have to use again $u$ and $w$. On the remaining two positions, we induce the maps using the universal properties of the cokernels and kernels that we have. These universal properties guarantee that $\delta_{\C}$ is a functor. Using the explicit chacterization of kernels and cokernels in the category of short exact sequences, that we proved in \prox\ref{propcharkerses}, applied to both $\Ses{\C}$ and $\Ses{\Ses{\C}}$, we obtain that $\delta_{\C}$ preserves kernels and cokernels. So $\delta_{\C}$ is a morphism in $\ClIdl$.
    
    Moreover, the components $\delta_{\C}$ can be organized into a pseudo-natural transformation $\delta\:\Omega\aR{}\Omega\c\Omega$ as follows. Given a morphism $G\:\C\to \D$ in $\ClIdl$, we define the structure 2-cell
    \begin{cd}[5]
        {\Ses{\C}} \& {\Ses{\Ses{\C}}} \\
	{\Ses{\D}} \& {\Ses{\Ses{\D}}}
	\arrow["{\delta_{\C}}", from=1-1, to=1-2]
	\arrow[""{name=0, anchor=center, inner sep=0}, "G"', from=1-1, to=2-1]
	\arrow[""{name=1, anchor=center, inner sep=0}, "G", from=1-2, to=2-2]
	\arrow["{\delta_{\D}}"', from=2-1, to=2-2]
	\arrow["{\delta_{G}}"', iso, from=0, to=1]
    \end{cd}
    to be the natural transformation that has component on a short exact sequence $X\ar{f} Y \ar{g} Z$ in $\C$ given by
    \begin{cd}[4]
        \&\&\&[-6ex] {G(X)} \& {G(X)} \& {C(\id{G(X)})} \\
	{G(X)} \& {G(X)} \& {G(C(\id{X}))} \& {G(X)} \& {G(Y)} \& {G(Z)} \\
	{G(X)} \& {G(Y)} \& {G(Z)} \& {K(\id{G(Z)})} \& {G(Z)} \& {G(Z)} \\
	{G(K(\id{Z}))} \& {G(Z)} \& {G(Z)}
	\arrow[equals, from=1-4, to=1-5]
	\arrow[equals, from=1-4, to=2-4]
	\arrow["{\coker{\id{G(X)}}}", from=1-5, to=1-6]
	\arrow["{G(f)}", from=1-5, to=2-5]
	\arrow["{G(w)}", from=1-6, to=2-6]
	\arrow[equals, from=2-1, to=2-2]
	\arrow[equals, from=2-1, to=3-1]
	\arrow["{G(\coker{\id{X}})}", from=2-2, to=2-3]
	\arrow["{G(f)}", from=2-2, to=3-2]
	\arrow[aiso,from=2-3, to=1-6]
	\arrow["{G(w)}", from=2-3, to=3-3]
	\arrow["{G(f)}", from=2-4, to=2-5]
	\arrow["{G(u)}"', from=2-4, to=3-4]
	\arrow["{G(g)}", from=2-5, to=2-6]
	\arrow["{G(g)}", from=2-5, to=3-5]
	\arrow[equals, from=2-6, to=3-6]
	\arrow["{G(f)}", from=3-1, to=3-2]
	\arrow["{G(u)}"', from=3-1, to=4-1]
	\arrow["{G(g)}", from=3-2, to=3-3]
	\arrow["{G(g)}", from=3-2, to=4-2]
	\arrow[equals, from=3-3, to=4-3]
	\arrow["{\ker{\id{G(Z)}}}"', from=3-4, to=3-5]
	\arrow[equals, from=3-5, to=3-6]
	\arrow[aiso,from=4-1, to=3-4]
	\arrow["{G(\ker{\id{Z}})}"', from=4-1, to=4-2]
	\arrow[equals, from=4-2, to=4-3]
    \end{cd}
    where the morphisms we have not drawn are all identities and the isomorphisms we have drawn are given by the fact that $G$ preserves kernels and cokernels. It is easy to see that this is indeed an isomorphism in $\Ses{\Ses{\D}}$. Notice that the two isomorphisms drawn in the picture are the unique morphisms that can be placed in that position, since $\ker{\id{G(Z)}}$ is mono and $G(\coker{\id{X}})$ is epi. This guarantees that $\delta$ is a pseudo-natural transformation.

    We prove that $(\Omega,\delta,\epsilon)$ extend to a pseudo-comonad on $\ClIdl$. Notice first that the following two triangles of pseudo-natural transformations commute:
    \begin{eqD*}
    \begin{cd}*
        \Omega \& {\Omega\c \Omega} \\
	\& \Omega
	\arrow["\delta", Rightarrow, from=1-1, to=1-2]
	\arrow[bend right=20, equals, from=1-1, to=2-2]
	\arrow["{\epsilon \Omega}", Rightarrow, from=1-2, to=2-2]
    \end{cd} \qquad
    \begin{cd}*
        \Omega \& {\Omega\c \Omega} \\
	\& \Omega
	\arrow["\delta", Rightarrow, from=1-1, to=1-2]
	\arrow[bend right=20, equals, from=1-1, to=2-2]
	\arrow["{\Omega \epsilon}", Rightarrow, from=1-2, to=2-2]
    \end{cd}
    \end{eqD*}
    This is because the action of $\delta_{\C}$ places the input (without modifying it) on both the middle row and the middle column of the $3\times 3$ grid. And also the structure 2-cells $\delta_G$ just have identities over the middle row and the middle column. We then construct an invertible modification
    \begin{cd}[6.5]
    \Omega \& {\Omega\c \Omega} \\
	{\Omega\c \Omega} \& {\Omega\c \Omega\c \Omega}
	\arrow["\delta", Rightarrow, from=1-1, to=1-2]
	\arrow[""{name=0, anchor=center, inner sep=0}, "\delta"', Rightarrow, from=1-1, to=2-1]
	\arrow[""{name=1, anchor=center, inner sep=0}, "{\Omega \delta}", Rightarrow, from=1-2, to=2-2]
	\arrow["{\delta \Omega}"', Rightarrow, from=2-1, to=2-2]
	\arrow["\Xi"'{inner sep=1.1ex}, iso, from=0, to=1]
    \end{cd}
    We define its component on $\C\in \ClIdl$ to be the natural isomorphism 
    \begin{cd}
    {\Ses{\C}} \& {\Ses{\Ses{\C}}} \\
	{\Ses{\Ses{\C}}} \& {\Ses{\Ses{\Ses{\C}}}}
	\arrow["{\delta_{\C}}", Rightarrow, from=1-1, to=1-2]
	\arrow[""{name=0, anchor=center, inner sep=0}, "{\delta_{\C}}"', Rightarrow, from=1-1, to=2-1]
	\arrow[""{name=1, anchor=center, inner sep=0}, "{\delta_{\C}}", Rightarrow, from=1-2, to=2-2]
	\arrow["{\delta_{\Ses{\C}}}"', Rightarrow, from=2-1, to=2-2]
	\arrow["{\Xi_{\C}}"'{inner sep=1.1ex}, iso, from=0, to=1]
    \end{cd}
    that has component on a short exact sequence $X\ar{f} Y \ar{g} Z$ given by identities in the twenty-three appropriate positions and four more interesting morphisms in positions $(3,1,1)$, $(3,3,1)$, $(1,1,3)$, $(1,3,3)$, of which the former two are dual to the latter two. In position $(3,3,1)$ we put the canonical isomorphism
    $$C(\id{X})\iso C(\id{C(\id{X})})$$
    that is given by the fact that $C(\id{X})$ is a null object and $\lemx\ref{lemmino}$. While in position $(3,1,1)$ we put the canonical isomorphism
    $$K(\id{C(\id{X})})\iso K(\coker{\id{C(\id{X})}})$$
    that is induced by the isomorphism above. It is straightforward to prove that $\Xi$ is a well-defined modification.

    It is now straightforward to prove that $(\Omega,\delta,\epsilon,\Xi)$ is a pseudo-comonad.
  
    We now prove that the 2-category of pseudo-coalgebras for $\Omega$ is 2-isomorphic to the 2-category $\BiHPTors$, over $\ClIdl$.

    A pseudo-coalgebra for $\Omega$ is $(\C,\N)\in \ClIdl$ equipped with a coalgebra morphism $\lambda\:\C\to \Ses{\C}$ in $\ClIdl$ and a natural isomorphism 
    \begin{cd}
        \C \& {\Ses{\C}} \\
	{\Ses{\C}} \& {\Ses{\Ses{\C}}}
	\arrow["\lambda", from=1-1, to=1-2]
	\arrow[""{name=0, anchor=center, inner sep=0}, "\lambda"', from=1-1, to=2-1]
	\arrow[""{name=1, anchor=center, inner sep=0}, "\lambda", from=1-2, to=2-2]
	\arrow["{\delta_{\C}}"', from=2-1, to=2-2]
	\arrow["{\lambda_{\delta}}"'{inner sep = 1ex}, iso, from=0, to=1]
    \end{cd}
    such that the triangle
    \begin{cd}[5]
        \C \& {\Ses{\C}} \\
	\& \C
	\arrow["\lambda", from=1-1, to=1-2]
	\arrow[bend right=20, equals, from=1-1, to=2-2]
	\arrow["{\epsilon_{\C}}", from=1-2, to=2-2]
    \end{cd}
    commutes and the axioms of pseudo-coalgebra are satisfied. Thanks to the triangle above, the coalgebra map $\lambda$ thus assigns to every $X\in \C$ a short exact sequence
    $$T^X\ar{\ell^x} X \ar{r^X}F^X$$
    with $X$ in the middle, and to every morphism $h\:X\to Y$ in $\C$ a morphism in $\Ses{\C}$
    \begin{cd}[5]
        {T^X} \& X \& {F^X} \\
	{T^Y} \& Y \& {F^Y}
	\arrow["{\ell^X}", from=1-1, to=1-2]
	\arrow["{h^T}"', from=1-1, to=2-1]
	\arrow["{r^X}", from=1-2, to=1-3]
	\arrow["h", from=1-2, to=2-2]
	\arrow["{h^F}", from=1-3, to=2-3]
	\arrow["{\ell^Y}"', from=2-1, to=2-2]
	\arrow["{r^Y}"', from=2-2, to=2-3]
    \end{cd}
    with $h$ in the middle.

    Moreover, for every $X\in \C$, $\lambda_{\delta}$ provides an isomorphism in $\Ses{\Ses{\C}}$ given by
    \begin{cd}[4][6.8]
    \&\&\&[-3ex] {T^X} \& {T^X} \& {C(\id{T^X})} \\
	{T^{T^X}} \& {T^X} \& {F^{T^X}} \& {T^X} \& X \& {F^X} \\
	{T^X} \& X \& {F^X} \& {K(\id{F^X})} \& {F^X} \& {F^{F^X}} \\
	{T^{F^X}} \& {F^X} \& {F^{F^X}}
	\arrow[equals, from=1-4, to=1-5]
	\arrow[equals, from=1-4, to=2-4]
	\arrow["{\coker{\id{T^X}}}", from=1-5, to=1-6]
	\arrow["{\ell^X}", from=1-5, to=2-5]
	\arrow["w", from=1-6, to=2-6]
	\arrow["{\lambda_{\delta}^{1,1}}"{description}, from=2-1, to=1-4]
	\arrow["{\ell^{T^X}}", from=2-1, to=2-2]
	\arrow["{(\ell^X)^T}"', from=2-1, to=3-1]
	\arrow["{r^{T^X}}", from=2-2, to=2-3]
	\arrow["{\ell^X}", from=2-2, to=3-2]
	\arrow["{\lambda_{\delta}^{1,3}}"{description}, from=2-3, to=1-6]
	\arrow["{(\ell^X)^F}", from=2-3, to=3-3]
	\arrow["{\ell^X}", from=2-4, to=2-5]
	\arrow["u"', from=2-4, to=3-4]
	\arrow["{r^X}", from=2-5, to=2-6]
	\arrow["{r^X}"', from=2-5, to=3-5]
	\arrow[equals, from=2-6, to=3-6]
	\arrow["{\ell^X}", from=3-1, to=3-2]
	\arrow["{(r^X)^T}"', from=3-1, to=4-1]
	\arrow["{r^X}", from=3-2, to=3-3]
	\arrow["{r^X}"', from=3-2, to=4-2]
	\arrow["{(r^X)^T}", from=3-3, to=4-3]
	\arrow["{\ker{\id{F^X}}}"', from=3-4, to=3-5]
	\arrow[equals, from=3-5, to=3-6]
	\arrow["{\lambda_{\delta}^{3,1}}"{description}, from=4-1, to=3-4]
	\arrow["{\ell^{F^X}}"', from=4-1, to=4-2]
	\arrow["{r^{F^X}}"', from=4-2, to=4-3]
	\arrow["{\lambda_{\delta}^{3,3}}"{description}, from=4-3, to=3-6]
    \end{cd}
    The axiom of pseudo-coalgebra that links $\lambda_{\delta}$ with the counit (together with its dual axiom that is automatically satisfied) precisely translates into the request that all the morphisms we have not drawn above are identities. It is thus implied that
    $$T^{T^X}\aiso{\lambda_{\delta}^{1,1}} T^X \quad \text{ and } \quad F^{T^X} \text{ null object}$$
    $$F^{F^X}\aiso{\lambda_{\delta}^{3,3}} F^X \quad \text{ and } \quad T^{F^X} \text{ null object}$$

    The naturality of $\lambda_{\delta}$ translates into having, for every morphism $h\:X\to Y$ in $\C$ the two axioms
    \begin{eqD*}
        \begin{cd}*
        {T^{T^X}} \& {T^{T^Y}} \\
	{T^{X}} \& {T^{Y}}
	\arrow["{(h^T)^T}", from=1-1, to=1-2]
	\arrow[from=1-1, to=2-1, aiso]
	\arrow[from=1-2, to=2-2, aiso]
	\arrow["{h^T}"', from=2-1, to=2-2]
        \end{cd}
        \qquad
        \begin{cd}*
            {F^{T^X}} \& {F^{T^Y}} \\
	{\coker{\id{T^X}}} \& {\coker{\id{T^Y}}}
	\arrow["{(h^T)^F}", from=1-1, to=1-2]
	\arrow[from=1-1, to=2-1, aiso]
	\arrow[from=1-2, to=2-2, aiso]
	\arrow["{\exists !}"', from=2-1, to=2-2]
        \end{cd}
    \end{eqD*}
    plus their dual (on the middle row and middle column of the $3\times 3$ grid we have only trivial axioms).

    The remaining axiom of pseudo-coalgebra that $\lambda_{\delta}$ needs to satisfy translates into twenty-seven diagrams, of which only the following eight and their duals are interesting:
    \begin{eqD*}
        \begin{cd}*
            {T^{T^{T^X}}} \& {T^{{T^X}}} \& {{{T^X}}} \\
	{T^{{T^X}}} \& {T^{{T^X}}} \& {{{T^X}}}
	\arrow["{(\lambda^{1,1}_{\delta})_{T^X}}", from=1-1, to=1-2]
	\arrow["{(\lambda^{1,1}_{\delta})_{X}^{T}}"', from=1-1, to=2-1]
	\arrow["{(\lambda^{1,1}_{\delta})_{X}}", from=1-2, to=1-3]
	\arrow[from=1-3, to=2-3, equal]
	\arrow[from=2-1, to=2-2, equal]
	\arrow["{(\lambda^{1,1}_{\delta})_{T^X}}"', from=2-2, to=2-3]
        \end{cd}
    \qquad
    \begin{cd}*
        {{T^{T^X}}} \& {T^{{T^X}}} \& {{{T^X}}} \\
	{{{T^X}}} \& {{{T^X}}} \& {{{T^X}}}
	\arrow[equals, from=1-1, to=1-2]
	\arrow["{(\lambda^{1,1}_{\delta})_{X}}"', from=1-1, to=2-1]
	\arrow["{(\lambda^{1,1}_{\delta})_{X}}", from=1-2, to=1-3]
	\arrow[equals, from=1-3, to=2-3]
	\arrow[equals, from=2-1, to=2-2]
	\arrow["{(\lambda^{1,1}_{\delta})_{T^X}}"', from=2-2, to=2-3]
    \end{cd}
    \end{eqD*}
    \begin{eqD*}
        \begin{cd}*
            {F^{T^{T^X}}} \& {\coker{\id{T^{T^X}}}} \& {\coker{\id{{T^X}}}} \\
	{F^{{T^X}}} \& {F^{{T^X}}} \& {\coker{\id{{T^X}}}}
	\arrow["{(\lambda^{1,3}_{\delta})_{T^X}}", from=1-1, to=1-2]
	\arrow["{(\lambda^{1,3}_{\delta})_{X}^F}"', from=1-1, to=2-1]
	\arrow["{\exists !}", from=1-2, to=1-3]
	\arrow[equals, from=1-3, to=2-3]
	\arrow[equals, from=2-1, to=2-2]
	\arrow["{(\lambda^{1,3}_{\delta})_{X}}"', from=2-2, to=2-3]
        \end{cd}
        \begin{cd}*
           {{T^{T^X}}} \& {{T^{T^X}}} \& {{{T^X}}} \\
	{T^{{T^X}}} \& {T^{{T^X}}} \& {{{T^X}}}
	\arrow[equals, from=1-1, to=1-2]
	\arrow["{(\lambda^{1,1}_{\delta})_{X}^T}"', from=1-1, to=2-1]
	\arrow["{(\lambda^{1,1}_{\delta})_{T^X}}", from=1-2, to=1-3]
	\arrow[equals, from=1-3, to=2-3]
	\arrow[equals, from=2-1, to=2-2]
	\arrow["{(\lambda^{1,1}_{\delta})_{X}}"', from=2-2, to=2-3]
        \end{cd}
    \end{eqD*}
    \begin{eqD*}
        \begin{cd}*
           {{F^{T^X}}} \& {{F^{T^X}}} \& {\coker{\id{T^X}}} \\
	{{F^{T^X}}} \& {{F^{T^X}}} \& {\coker{\id{T^X}}}
	\arrow[equals, from=1-1, to=1-2]
	\arrow[equals, from=1-1, to=2-1]
	\arrow["{(\lambda^{1,3}_{\delta})_{X}}", from=1-2, to=1-3]
	\arrow[equals, from=1-3, to=2-3]
	\arrow[equals, from=2-1, to=2-2]
	\arrow["{(\lambda^{1,3}_{\delta})_{X}}"', from=2-2, to=2-3]
        \end{cd}
        \begin{cd}*
           {T^{F^{T^X}}} \& {\ker{\id{F^{T^{X}}}}} \& {\ker{\id{\coker{\id{T^X}}}}} \\
	{T^{\coker{\id{T^X}}}} \& {\ker{\nu}} \& {\coker{\id{\coker{\id{X}}}}}
	\arrow["{(\lambda^{3,1}_{\delta})_{X}}", from=1-1, to=1-2]
	\arrow["{(\lambda^{1,3}_{\delta})^T_{X}}"', from=1-1, to=2-1]
	\arrow["{\exists !}", from=1-2, to=1-3]
	\arrow["\Xi", from=1-3, to=2-3]
	\arrow["{\delta_{\lambda}}"', from=2-1, to=2-2]
	\arrow["{\exists !}"', from=2-2, to=2-3]
        \end{cd}
    \end{eqD*}
    \begin{eqD*}
        \begin{cd}*
           {{F^{T^X}}} \& {{F^{T^X}}} \& {{\coker{\id{T^X}}}} \\
	{{\coker{\id{T^X}}}} \& {{\coker{\id{T^X}}}} \& {{\coker{\id{T^X}}}}
	\arrow[equals, from=1-1, to=1-2]
	\arrow["{(\lambda^{1,3}_{\delta})_{X}}"', from=1-1, to=2-1]
	\arrow["{(\lambda^{1,3}_{\delta})_{X}}", from=1-2, to=1-3]
	\arrow[equals, from=1-3, to=2-3]
	\arrow[equals, from=2-1, to=2-2]
	\arrow[equals, from=2-2, to=2-3]
        \end{cd}
        \end{eqD*}
        \begin{eqD*}
        \begin{cd}*
           {F^{F^{T^X}}} \& {{F^{T^X}}} \& {{\coker{\id{T^X}}}} \\
	{F^{\coker{\id{T^X}}}} \& {{\coker{\id{F^{T^X}}}}} \& {{\coker{\id{\coker{\id{T^X}}}}}}
	\arrow["{(\lambda^{3,3}_{\delta})_{T^X}}", from=1-1, to=1-2]
	\arrow["{(\lambda^{1,3}_{\delta})^F_{X}}"', from=1-1, to=2-1]
	\arrow["{(\lambda^{1,3}_{\delta})_{X}}", from=1-2, to=1-3]
	\arrow["\Xi", from=1-3, to=2-3]
	\arrow["{{\delta_{\lambda}}}"', from=2-1, to=2-2]
	\arrow["{\exists !}"', from=2-2, to=2-3]
        \end{cd}
    \end{eqD*}
Here $\nu$ is the unique morphism induced by the universal property of $\coker{\id{F^{T^X}}}$ applied to $\coker{\id{F^{T^X}}} \c r^{T^X}$. 
    We prove that every pseudo-coalgebra $(\C,\N,\lambda,\lambda_{\delta})$ for $\Omega$ yields a bihereditary pretorsion theory. We define the classes of torsion and torsion-free objects to be
    $$\T= \lambda^{-1}(\T_{\Ses{\C}}) \quad \text{ and }\quad \F= \lambda^{-1}(\F_{\Ses{\C}})$$
    where $\T_{\Ses{\C}}$ and $\F_{\Ses{\C}}$ are the torsion and torsion-free classes of the pretorsion theory on $\Ses{\C}$ constructed in \prox\ref{proppretorsses}. Since $\T_{\Ses{\C}}$ and $\F_{\Ses{\C}}$ are full replete subcategories of $\Ses{\C}$, we obtain that $\T$ and $\F$ are full replete subcategories of $\C$. It is easy to show that $\T\cap \F= \N$.
    
    \noindent Moreover, every object $X\in \C$ has an associated short exact sequence $T^X\ar{\ell^X} X\ar{r^X} F^X$ with $X$ in the middle, given by $\lambda(X)$. We have that $T^X\in \T$, because
    $$\lambda(T^X)=(T^{T^X}\ar{\ell^{T^X}} T^X\ar{r^{T^X}} F^{T^X})$$
    is of the form $\bullet \iso \bullet \to \bullet$ thanks to the natural isomorphism $\lambda_{\delta}$. Dually, $F^X\in \F$.

    \noindent It remains to prove that every morphism $h$ in $\C$ from $A\in \T$ to $B\in \F$ is null. $\lambda$ sends $h$ to
    \begin{cd}
        {T^A} \& A \& {F^A} \\
	   {T^B} \& B \& {F^B}
	\arrow[iso, from=1-1, to=1-2]
	\arrow["{h^T}"', from=1-1, to=2-1]
	\arrow["{r^A}", from=1-2, to=1-3]
	\arrow["h", from=1-2, to=2-2]
	\arrow["{h^F}", from=1-3, to=2-3]
	\arrow["{\ell^B}"', from=2-1, to=2-2]
	\arrow[iso, from=2-2, to=2-3]
    \end{cd}
    Since $T^B$ is a null object, by the axioms of pseudo-coalgebra, we deduce that $h$ is a null morphism. We conclude that $(\C,\T,\F)$ is a pretorsion theory.

    \noindent We prove that it is a hereditary pretorsion theory. So consider a kernel $K(g)\ar{\ker{g}} X\ar{g} Y$ in $\C$. We need to show that the induced torsion part 
    $$T^{K(g)} \ar{\ker{g}^T} T^X \ar{g^T} T^Y$$
    is again a kernel in $\C$. Since $\lambda$ is a morphism in $\ClIdl$, it preserves kernels. So $(\ker{g}^T, \ker{g}, \ker{g}^F)$ is the kernel of $(g^T,g,g^F)$ in $\Ses{\C}$. Using \prox\ref{propcharkerses}, we obtain that $\ker{g}^T$ is the kernel of $g^T$ in $\C$, and we conclude. Dually, $(\C,\T,\F)$ is also cohereditary.

    We now prove that every bihereditary pretorsion theory $(\C,\T,\F)$ yields a pseudo-coalgebra for $\Omega$. Of course, $(\C,\T\cap \F)\in \ClIdl$. We then define
    \begin{fun}
	\lambda & \: & \C \hphantom{C} & \too & \hphantom{CC}\Ses{\C} \\[0.3ex]
    && \fib{X}{h}{Y} & \mto & \begin{cd}*[4][5.5]
{T^X} \& X \& {F^X} \\
	{T^Y} \& Y \& {F^Y}
	\arrow["{\ell^X}", from=1-1, to=1-2]
	\arrow["{h^T}"', from=1-1, to=2-1]
	\arrow["{r^X}", from=1-2, to=1-3]
	\arrow["h", from=1-2, to=2-2]
	\arrow["{h^F}", from=1-3, to=2-3]
	\arrow["{\ell^Y}"', from=2-1, to=2-2]
	\arrow["{r^Y}"', from=2-2, to=2-3]
    \end{cd}
    \end{fun}
    to send every $X$ to its associated short exact sequence given by the pretorsion theory and every $h\:X\to Y$ to the triple $(h^T,h,h^F)$ where $h^T$ is the torsion part of $h$ and $h^F$ is the torsion-free part of $h$ (induced by the universal properties of the relevant kernels and cokernels). Of course, $\lambda$ is a functor, by the uniqueness part of the universal properties. Moreover, $\lambda$ preserves kernels thanks to \prox\ref{propcharkerses} and the fact that $(\C,\T,\F)$ is hereditary. Dually, $\lambda$ preserves cokernels as well. So that $\lambda$ is a morphism in $\ClIdl$. By construction, $\epsilon_{\C}\c \lambda=\Id{}$.

    \noindent We define $\lambda_{\delta}$ to have component on $X\in \C$ given by identities in all the needed positions and the following two morphisms plus their duals. In position $(1,1)$, we put the canonical isomorphism $T^{T^X}\iso T^X$ that is given by the fact that $T^X\in \T$. And in position $(1,3)$ we put the canonical isomorphism $F^{T^X}\iso C(\id{T^X})$ obtained by the fact that $T^X \aeqq{} T^X \to C(\id{T^X})$ is another short exact sequence with $T^X$ in the middle other than the chosen one that has a torsion object on the left and a torsion-free object (in fact, a null object) on the right. It is straightforward to see that this gives a natural isomorphism $\lambda_{\delta}$. Notice that $\ell^{T^X}=(\ell^X)^T$, since $\ell^X$ is mono. And then also $\ell^{T^{T^X}}=(\ell^{T^X})^T$. It is now straightforward to prove that $(\C,\T\cap \F,\lambda,\lambda_{\delta})$ is a pseudo-comonad for $\Omega$. 

A pseudo-morphism of normal pseudo-$\Omega$-coalgebras from $(\C,\N, \lambda\: \C \to \ses{\C}, \lambda_\delta)$ to $(\D,\N', \mu\: \D \to \ses{\D}, \mu_\delta)$ is a pair $(G, \overline{G})$, where $G\: \C \to \D$ is a functor that preserves kernels and cokernels and $\overline{G}$ is an isomorphic  natural transformation 
\begin{cd}
    \C \& \D \\
	{\ses{\C}} \& {\ses{\D}.}
	\arrow["G", from=1-1, to=1-2]
	\arrow[""{name=0, anchor=center, inner sep=0}, "\lambda"', from=1-1, to=2-1]
	\arrow[""{name=1, anchor=center, inner sep=0}, "\mu", from=1-2, to=2-2]
	\arrow["{G}"', from=2-1, to=2-2]
	\arrow["{\overline{G}}"'{inner sep=1ex}, iso, from=0, to=1]
\end{cd}
So the component of $\overline{G}$ on $X\in \C$ is an isomorphism in $\ses{\D}$
\begin{cd}
    {T^{G(X)}} \& {G(X)} \& {F^{G(X)}} \\
	{G(T^X)} \& {G(X)} \& {G(F^X)}
	\arrow["{\ell^{G(X)}}", from=1-1, to=1-2]
	\arrow["{{\overline{G}}^1_X}"',aiso, from=1-1, to=2-1]
	\arrow["{r^{G(X)}}", from=1-2, to=1-3]
	\arrow["{{\overline{G}}^2_X}"',aiso,  from=1-2, to=2-2]
	\arrow["{{\overline{G}}^3_X}"',aiso,  from=1-3, to=2-3]
	\arrow["{G(\ell^X)}"', from=2-1, to=2-2]
	\arrow["{G(r^X)}"', from=2-2, to=2-3]
\end{cd}
The axiom on the compatibility with comultiplication written on components is an equality between morphisms in $\ses{\ses{\D}}$ and thus it consists of 9 commutative diagrams in $\D$. The diagram corresponding to the object in the middle in the short exact sequence, in position 2,2 in matrix notation, is trivial. And out of the other eight, four are dual to the other four. For instance, the one corresponding to position 3,2 is dual to the one of position 1,2. So we only explicitly write the ones corresponding to the positions 1,1 1,2 2,1 and 1,3. They are the following commutative diagrams
\begin{eqD*}
    \begin{cd}*
 {T^{G(X)}} \& {T^{T^{G(X)}}} \& {T^{G(T^X)}} \\
	{G(T^X)} \& {G(T^X)} \& {G(T^{T^{X}})}
	\arrow["{(\mu_{\delta}^{1,1})^{-1}_{G(X)}}", from=1-1, to=1-2]
	\arrow["{\overline{G}^1_X}"', from=1-1, to=2-1]
	\arrow["{(\overline{G}^1_X)^T}", from=1-2, to=1-3]
	\arrow["{\overline{G}^1_{T^X}}", from=1-3, to=2-3]
	\arrow[from=2-1, to=2-2, equal]
	\arrow["{G((\lambda_{\delta}^{1,1})^{-1}_X)}"', from=2-2, to=2-3]
\end{cd}
\qquad
\begin{cd}*
    {T^{G(X)}} \& {T^{{G(X)}}} \& {G(T^X)} \\
	{G(T^X)} \& {G(T^X)} \& {G({T^{X}})}
	\arrow["{(\mu_{\delta}^{1,2})^{-1}_{G(X)}}", from=1-1, to=1-2]
	\arrow["{\overline{G}^1_X}"', from=1-1, to=2-1]
	\arrow["{\overline{G}^1_X}", from=1-2, to=1-3]
	\arrow["{\overline{G}^2_{T^X}}", from=1-3, to=2-3]
	\arrow[from=2-1, to=2-2, equal]
	\arrow["{G((\lambda_{\delta}^{1,2})^{-1}_X)}"', from=2-2, to=2-3]
\end{cd}
\end{eqD*}
\begin{eqD*}
    \begin{cd}*
    {T^{G(X)}} \& {T^{G(X)}} \& {T^{G(X)}} \\
	{G(T^X)} \& {G(T^X)} \& {G({T^{X}})}
	\arrow["{(\mu_{\delta}^{2,1})^{-1}_{G(X)}}", from=1-1, to=1-2]
	\arrow["{\exists !}"', from=1-1, to=2-1]
	\arrow["{(\overline{G}^2_X)^T}", from=1-2, to=1-3]
	\arrow["{\overline{G}^1_{X}}", from=1-3, to=2-3]
	\arrow[from=2-1, to=2-2,equal]
	\arrow["{G((\lambda_{\delta}^{2,1})^{-1}_X)}"', from=2-2, to=2-3]
    \end{cd}
    \quad
     \begin{cd}*
     {\coker{\id{T^{G(X)}}}} \& {F^{T^{G(X)}}} \& {F^{G(T^{X})}} \\
	{\coker{\id{G(T^X)}}} \& {G(\coker{\id{T^X}})} \& {G(F^{T^{X}})}
	\arrow["{(\mu_{\delta}^{1,3})^{-1}_{G(X)}}", from=1-1, to=1-2]
	\arrow["{\exists !}"', from=1-1, to=2-1]
	\arrow["{(\overline{G}^1_X)^F}", from=1-2, to=1-3]
	\arrow["{\overline{G}^3_{T^X}}", from=1-3, to=2-3]
	\arrow["{\delta^{1,3}_G}"', from=2-1, to=2-2]
	\arrow["{G((\lambda_{\delta}^{1,3})^{-1}_X)}"', from=2-2, to=2-3]
    \end{cd}
\end{eqD*}
The other axiom of pseudo-morphism of normal pseudo-$\Omega$-coalgebras on component $X$ gives the condition $\overline{G}^2_X= \id{G(X)}$.
We now show that the functor $G:\C \to \D$ yields a functor between pretorsion theories $G\: (\C, \T, \F) \to (\D, \T', \F')$. We already know that $G$ preserves kernels and cokernels, so it remains to prove that it preserves torsion and torsion-free class. 
Let $A\in \T$ and consider the morphism of pretorsion theories $\overline{G}_A$:
\begin{cd}
    {T^{G(A)}} \& {G(A)} \& {F^{G(A)}} \\
	{G(T^A)} \& {G(A)} \& {G(F^A)}
	\arrow["{\ell^{G(A)}}", from=1-1, to=1-2]
	\arrow["{{\overline{G}}^1_A}"',aiso, from=1-1, to=2-1]
	\arrow["{r^{G(A)}}", from=1-2, to=1-3]
	\arrow[equal,  from=1-2, to=2-2]
	\arrow["{{\overline{G}}^3_A}"',aiso,  from=1-3, to=2-3]
	\arrow["{G(\ell^A)}"', from=2-1, to=2-2, aiso]
	\arrow["{G(r^A)}"', from=2-2, to=2-3]
\end{cd}
This gives an isomorphism between $G(A)$ and $T^{G(A)}\in \T'$. Since $\T'$ is a replete subcategory of $\D$, we conclude that $G(A)\in \T'$. Analogously, given $B\in \F$, we have $G(B)\in \F'$. So $G$ is a functor between pretorsion theories.

We now show that every functor between pretorsion theories $G\: (\C, \T, \F) \to (\D, \T', \F')$ (with $\C$ and $\D$ admitting all kernels and cokernels) has an associated pseudo-morphism of normal pseudo-$\Omega$-coalgebras 
$$(G, \overline{G})\:(\C,\N, \lambda\: \C \to \ses{\C}, \lambda_\delta)\to (\D,\N', \mu\: \D \to \ses{\D}, \mu_\delta).$$
$G$ is a functor that preserves kernels and cokernels by hypothesis. We now construct the natural transformation $\overline{G}$. Given $X\in \C$, the component of $\overline{G}$ on $X$ is an isomorphism in $\ses{\D}$ of the form
\begin{cd}
    {T^{G(X)}} \& {G(X)} \& {F^{G(X)}} \\
	{G(T^X)} \& {G(X)} \& {G(F^X)}
	\arrow["{\ell^{G(X)}}", from=1-1, to=1-2]
	\arrow["{{\overline{G}}^1_X}"',aiso, from=1-1, to=2-1]
	\arrow["{r^{G(X)}}", from=1-2, to=1-3]
	\arrow[equal,  from=1-2, to=2-2]
	\arrow["{{\overline{G}}^3_X}"',aiso,  from=1-3, to=2-3]
	\arrow["{G(\ell^X)}"', from=2-1, to=2-2]
	\arrow["{G(r^X)}"', from=2-2, to=2-3]
\end{cd}
Since $G$ preserves the torsion class, we have that $G(T^X)\in \T'$. This implies that the composite 
$$G(T^X) \ar{G(\ell^X)} G(X) \ar{r^{G(X)}} F^{G(X)}$$
is null, since it is a morphism from an object of $\T'$ to an object of $\F'$. So there exists a unique morphism $G(T^X) \ar{\xi} T^{G(X)}$ such that $\ell^{G(X)} \c  \xi=G(\ell^{X})$. We show that such $\xi$ is an isomorphism. Since $G$ preserves kernels, we have that $G(\ell^X)$ is the kernel of $G(r^X)$ and thus 
$$G(T^X) \ar{G(\ell^X)} G(X) \ar{G(r^X)} G(F^X)$$
is a short exact sequence. Moreover, $G(r^X) \c \ell^{G(X)}$ is null because $G(F^X)\in \F'$. Whence, there exists a unique morphism $\xi'\: T^{G(X)} \to G(T^X)$ such that $G(\ell^X) \c \xi'= \ell^{G(X)}$. Since $G(\ell^X)$ is a monomorphism, we have $\xi' \c \xi=\id{T^{G(X)}}$. And since $\ell^{G(X)}$ is a monomorphism, we have $\xi \c \xi'=\id{G(T^X)}$. So $\xi$ is an isomorphism and we define $\overline{G}^1_X:=\xi$. We then define the isomorphism $\overline{G}^3_X$ dually. By construction, the triple $({G}^1_X, \id{G(X)}, {G}^3_X)$ is a morphism in $\ses{\D}$. 

We now show that $\overline{G}$ is natural. Let $h\: X \to Y$ be a morphism in $\C$. The naturality square for $h$ is the equality of the following two composite morphisms in $\ses{\D}$:
\begin{eqD*}
\begin{cd}*
{T^{G(X)}} \& {G(X)} \& {F^{G(X)}} \\
	{G(T^X)} \& {G(X)} \& {G({F^{X}})} \\
	{G(T^Y)} \& {G(Y)} \& {G(F^Y)}
	\arrow["{\ell^{G(X)}}", from=1-1, to=1-2]
	\arrow["{\overline{G}^1_X}"', from=1-1, to=2-1]
	\arrow["{r^{G(X)}}", from=1-2, to=1-3]
	\arrow[from=1-2, to=2-2, equal]
	\arrow["{\overline{G}^3_{X}}", from=1-3, to=2-3]
	\arrow["{G(\ell^X)}"', from=2-1, to=2-2]
	\arrow["{G(h^T)}"', from=2-1, to=3-1]
	\arrow["{G(r^X)}"', from=2-2, to=2-3]
	\arrow["{G(h)}", from=2-2, to=3-2]
	\arrow["{G(h^F)}", from=2-3, to=3-3]
	\arrow["{G(\ell^Y)}"', from=3-1, to=3-2]
	\arrow["{G(r^Y)}"', from=3-2, to=3-3]
\end{cd}
\qquad
\begin{cd}*
{T^{G(X)}} \& {G(X)} \& {F^{G(X)}} \\
	{T^{G(Y)}} \& {G(Y)} \& {F^{G(Y)}} \\
	{G(T^Y)} \& {G(Y)} \& {G(F^Y)}
	\arrow["{\ell^{G(X)}}", from=1-1, to=1-2]
	\arrow["{G(h)^T}"', from=1-1, to=2-1]
	\arrow["{r^{G(X)}}", from=1-2, to=1-3]
	\arrow["{G(h)}"', from=1-2, to=2-2]
	\arrow["{G(h)^F}", from=1-3, to=2-3]
	\arrow["{\ell^{G(Y)}}"', from=2-1, to=2-2]
	\arrow["{\overline{G}^1_Y}"', from=2-1, to=3-1]
	\arrow["{r^{G(Y)}}"', from=2-2, to=2-3]
	\arrow[from=2-2, to=3-2, equal]
	\arrow["{\overline{G}^3_Y}", from=2-3, to=3-3]
	\arrow["{G(\ell^Y)}"', from=3-1, to=3-2]
	\arrow["{G(r^Y)}"', from=3-2, to=3-3]
\end{cd}
\end{eqD*}
Since $G(\ell^Y)$ is a monomorphism, the first components of the two composite are equal. Moreover, the second components are trivially equal and the third components are equal thanks to the fact that $r^{G(X)}$ is an epimorphism. So $\overline{G}$ is natural.
It is then straightforward to see that the required axioms are satisfied.
So $(G, \overline{G})$ is a pseudo-morphism of normal pseudo-$\Omega$-coalgebras.

A 2-cell between the pseudo-morphisms of pseudo-coalgebras
$$(G, \overline{G}), (H, \overline{H})\:(\C,\N, \lambda\: \C \to \ses{\C}, \lambda_\delta)\to (\D,\N', \mu\: \D \to \ses{\D}, \mu_\delta)$$
is a natural transformation $\alpha: G \aR{} H$ such that for every $X\in \C$ the following equality of composites holds in $\ses{\D}$
\begin{eqD*}
    \begin{cd}*
    {T^{G(X)}} \& {G(X)} \& {F^{G(X)}} \\
	{T^{H(X)}} \& {H(X)} \& {F^{H(X)}} \\
	{H(T^X)} \& {H(X)} \& {H(F^X)}
	\arrow["{\ell^{G(X)}}", from=1-1, to=1-2]
	\arrow["{\alpha^T_X}"', from=1-1, to=2-1]
	\arrow["{r^{G(X)}}", from=1-2, to=1-3]
	\arrow["{\alpha_X}"', from=1-2, to=2-2]
	\arrow["{\alpha^F_X}", from=1-3, to=2-3]
	\arrow["{\ell^{H(X)}}"', from=2-1, to=2-2]
	\arrow["{\overline{H}^1_X}"', from=2-1, to=3-1]
	\arrow["{r^{H(X)}}"', from=2-2, to=2-3]
	\arrow[from=2-2, to=3-2, equal]
	\arrow["{\overline{H}^3_X}", from=2-3, to=3-3]
	\arrow["{H(\ell^X)}"', from=3-1, to=3-2]
	\arrow["{H(r^X)}"', from=3-2, to=3-3]
    \end{cd}
    \quad = \quad
    \begin{cd}*
    {T^{G(X)}} \& {G(X)} \& {F^{G(X)}} \\
	{G(T^X)} \& {G(X)} \& {G(F^X)} \\
	{H(T^X)} \& {H(X)} \& {H(F^X)}
	\arrow["{\ell^{G(X)}}", from=1-1, to=1-2]
	\arrow["{\overline{G}^1_X}"', from=1-1, to=2-1]
	\arrow["{r^{G(X)}}", from=1-2, to=1-3]
	\arrow[from=1-2, to=2-2, equal]
	\arrow["{\overline{G}^3_X}", from=1-3, to=2-3]
	\arrow["{G(\ell^{X})}"', from=2-1, to=2-2]
	\arrow["{\alpha_{T^X}}"', from=2-1, to=3-1]
	\arrow["{G(r^{X})}"', from=2-2, to=2-3]
	\arrow["{\alpha_{X}}", from=2-2, to=3-2]
	\arrow["{\alpha_{F^X}}", from=2-3, to=3-3]
	\arrow["{H(\ell^X)}"', from=3-1, to=3-2]
	\arrow["{H(r^X)}"', from=3-2, to=3-3]
    \end{cd}
\end{eqD*}
Clearly, every 2-cell between pseudo-morphisms of coalgebras yields a 2-cell between morphisms of pretorsion theories that is simply the underlying natural transformation. Conversely, every natural transformation $\alpha\:G \aR{} H$ between morphisms of bihereditary torsion theories satisfies the required axiom for a 2-cell between pseudo-morphisms of pseudo-coalgebras, because $H(\ell^X)$ is a monomorphism and $r^{G(X)}$ is an epimorphism for every $X\in \C$.

We have thus constructed two assignments 
 \begin{fun}
	\Theta & \: & \Alg+'{\Omega}\hphantom{CC} & \too & \BiHPTors \\[0.3ex]
    && \tcv*{(\C,\N, \lambda\: \C \to \ses{\C}, \lambda_\delta)}{(\D,\N', \mu\: \D \to \ses{\D}, \mu_\delta)}{(G, \overline{G})}{(H, \overline{H})}{\alpha} & \mto & \tcv*{(\C,\T,\F)}{(\D,\T',\F')}{G}{H}{[\alpha]}
    \end{fun}
and 
\begin{fun}
	\Gamma & \: & \BiHPTors & \too & \hphantom{CC}\Alg+'{\Omega} \\[0.3ex]
    && \tcv*{(\C,\T,\F)}{(\D,\T',\F')}{G}{H}{[\alpha]} & \mto & \tcv*{(\C,\N, \lambda\: \C \to \ses{\C}, \lambda_\delta)}{(\D,\N', \mu\: \D \to \ses{\D}, \mu_\delta)}{(G, \overline{G})}{(H, \overline{H})}{\alpha}
    \end{fun}
$\Theta$ is clearly a 2-functor by construction and it is straightforward to prove that also $\Gamma$ is a 2-functor. Furthermore, the composite $\Theta \c \Gamma$ is the identity functor. Indeed, the objects $A\in \C$ such that the short exact sequence $T^A \ar{\ell^A} A \ar{r^A} F^{A}$ is in $\T_{\ses{\C}}$ are precisely the objects of $\T$. So $\T=\lambda^{-1}(\T_{\ses{\C}})$. Dually, $\F=\lambda^{-1}(\F_{\ses{\C}})$. And it is straightforward to see that $\Theta \c \Gamma$ is the identity on morphisms and 2-cells.
 It is then straightforward to prove that composite $\Gamma \c \Theta$ is the identity functor as well. Notice that choosing the short exact sequences associated to every object uniquely determines the torsion and torsion-free parts of the morphisms. Moreover, from the axioms of pseudo-coalgebra, it follows that $\lambda_{\delta}$, on every component $X\in \C$, is uniquely determined by $\ell^{T^X}$ and $r^{F^X}$. Finally, the fact that $\Gamma \c \Theta$ is the identity on morphisms follow from the fact that there can only be one $\overline{G}$ associated to $G$, because $G(\ell^X)$ is a monomorphism and $r^{G(X)}$ is an epimorphism.
 
We thus conclude that the 2-category of pseudo-coalgebras for $\Omega$ is 2-isomorphic to the 2-category $\BiHPTors$ of bihereditary pretorsion theories, over $\ClIdl$. Notice that, although $\Omega$ is a pseudo-comonad, the pseudo-$\Omega$-coalgebras form a 2-category, thanks to the fact that $\Omega$ is a 2-functor and that given $G$ there is only at most one $\overline{G}$.
\end{proof}

\begin{remark}
    The pretorsion theory on $\Ses{\C}$ that we built in \ref{proppretorsses} is thus the cofree bihereditary pretorsion theory on $\C\in \ClIdl$.
\end{remark}

\begin{remark}\label{remideaextension}
    It is the fact that $\lambda\:\C \to \Ses{\C}$ preserves kernels and cokernels that implies that the pseudo-coalgebras for $\Omega$ are pretorsion theories with additional properties, namely bihereditary, rather than all pretorsion theories. Notice that $\lambda$ has to preserve kernels and cokernels because of the construction of $\ClIdl$. We defined $\ClIdl$ in that way to guarantee that a morphism $G\:\C\to \D$ in $\ClIdl$ yields a functor $\Ses{\C}\to \Ses{D}$ that preserves short exact sequences. Interestingly, this is actually the only point that relies on the choice that morphisms in $\ClIdl$ preserve kernels and cokernels. This observation opens the way for the extension of the pseudo-comonad $\Omega$ that we present in the following section.
\end{remark}

\section{A comonad for generalized pretorsion theories}

In this section, we extend the pseudo-comonad $(\Omega,\delta,\epsilon)$ that we built in \thex\ref{theorbih}. We obtain another pseudo-comonad on a 2-category of categories equipped with closed ideals. With this, we reach our aim: all pretorsion theories are coalgebras for this pseudo-comonad. But interestingly, pseudo-coalgebras correspond to a generalization of pretorsion theories.

\remx\ref{remideaextension} gives the idea to extend the 2-category $\ClIdl$ in the following way.

\begin{definition}
    Let $\C$ be a category equipped with a closed ideal that has all relative kernels and cokernels. An \dfn{exact sequence} in $\C$ is a null sequence of composable morphisms $X\ar{f} Y \ar{g} Z$ such that considering the induced diagram
    \begin{cd}[4]
    \& {K(g)} \\
	X \&\& Y \&\& Z \\
	\&\&\& {C(f)}
	\arrow["{\ker{g}}", from=1-2, to=2-3]
	\arrow["{\exists ! \xi^1}", dashed, from=2-1, to=1-2]
	\arrow["f"', from=2-1, to=2-3]
	\arrow["g", from=2-3, to=2-5]
	\arrow["{\coker{f}}"', from=2-3, to=3-4]
	\arrow["{\exists !\xi^2}"', dashed, from=3-4, to=2-5]
    \end{cd}
    we have that 
    $$K(g)\ar{\ker{g}} Y \ar{\coker{f}} C(f)$$
    is a short exact sequence, called the \dfn{short exact sequence replacement} of the original sequence, $\xi^1$ coreflects null morphisms and $\xi^2$ reflects null morphisms.
\end{definition}

\begin{remark}\label{remchoice}
    Given a category $\C$ equipped with a closed ideal that has all kernels and cokernels, we fix a choice, for every exact sequence in $\C$, of its short exact sequence replacement. For the exact sequences that are actually short exact, we choose the replacement to be exactly the original sequence.
\end{remark}

To better understand the definition above, recall the following known characterization of reflecting null morphisms.

\begin{proposition}\label{propcharactreflectnull}
    Let $\xi\:W\to Z$ be a morphism in a category $\C$ equipped with a closed ideal that has all kernels and cokernels. Then $\xi$ reflects null morphisms if and only if the kernel $K(\xi)$ of $\xi$ is a null object.
\end{proposition}

Moreover, recall that every kernel in a category with a closed ideal automatically reflects null morphisms, and dually for cokernels.

We can now extend the ground 2-category $\ClIdl$ that we considered in the previous section.

We define $\ClIdlex$ to be the 2-category given by the following:
\begin{description}
    \item[an object of $\ClIdlex$ is] an object of $\ClIdl$, i.e.\ a pair $(\C,\N)$ of a category $\C$ equipped with a closed ideal $\N$, that has all $\N$-kernels and all $\N$-cokernels;
    \item[a morphism from $(\C,\N)$ to $(\D,\M)$ is] a functor $F\:\C\to \D$ that preserves all exact sequences;
    \item[a 2-cell is] just a natural transformation between functors.
\end{description}

\begin{proposition}
    Every morphism in $\ClIdlex$ automatically preserves null objects, morphisms that reflect null morphisms and morphisms that coreflect null morphisms.

    Moreover, every morphism in $\ClIdl$ is a morphism in $\ClIdlex$.
\end{proposition}
\begin{proof}
    Let $G$ be a morphism in $\ClIdlex$ and let $Z$ be a null object. We have that $Z\aeqq{} Z\aeqq{} Z$ is a short exact sequence, and thus $G(Z)\aeqq{} G(Z) \aeqq{} G(Z)$ is an exact sequence. Since $K(\id{G(Z)})$ and $C(\id{G(Z)})$ are both null objects, we obtain that $G(Z)$ is a null object as well.

    Consider now a morphism $h\:X\to Y$ in $\C$ that reflects null morphisms. By \prox\ref{propcharactreflectnull}, $K(h)$ is a null object and thus the cokernel of $\ker{h}$ is the identity of $X$. This implies that
    $$K(h)\ar{\ker{h}}X\ar{h} Y$$
    is an exact sequence. Indeed, its short exact replacement is $K(h)\ar{\ker{h}} X\aeqq{} X$, and $h$ reflects null morphisms by assumption. Then $G$ preserves the exactness of such exact sequence. Since $G(K(h))$ is a null object, by what we proved above, the cokernel of $G(\ker{h})$ is the identity. And we conclude that $G(h)$ has to reflect null morphisms, by definition of exact sequence. Dually, $G$ preserves morphisms that coreflect null morphisms.

    Assume now that $G\:\C\to \D$ is a morphism in $\ClIdl$ and consider an exact sequence $X\ar{f} Y \ar{g} Z$ in $\C$. Call $\xi^1$ and $\xi^2$ the induced morphisms $X\to \ker{g}$ and $\coker{f}\to Z$ respectively. Then $G$ sends the short exact replacement
    $$K(g)\ar{\ker{g}} Y \ar{\coker{f}} C(f)$$
    of the exact sequence to a short exact sequence. Since $G$ preserves kernels and cokernels, we obtain that
    $$K(G(g))\ar{\ker{G(g)}} G(Y) \ar{\coker{G(f)}} C(G(f))$$
    is a short exact sequence. Finally, consider the induced ${\xi'}^1\:G(X)\to K(G(g))$ and ${\xi'}^2\:C(G(f))\to G(Z)$. Since $G(\coker(f))$ is a cokernel and thus epi, we have that the composite of ${\xi'}^2$ with the canonical isomorphism $G(C(f))\iso C(G(f))$ coincides with $G(\xi^2)$. Since the $\xi^2$ has null kernel and $G$ preserves kernels and null objects, we deduce that $G(\xi^2)$ has null kernel. But then also ${\xi'}^2$ has null kernel. Dually for ${\xi'}^1$.
\end{proof}

We now extend the pseudo-comonad $\Omega$ to the larger ground base $\ClIdlex$.

\begin{construction}\label{consomegaex}
    The 2-functor $\Omega$ extends to a pseudofunctor
    \begin{fun}
	\Oex & \: & \ClIdlex & \too & \ClIdlex \\[0.3ex]
    && \tcv*{\C}{\D}{G}{H}{\alpha} & \mto & \tcv*{\Ses{\C}}{\Ses{\D}}{[G]}{[H]}{[\alpha]}
\end{fun}
On objects, the action of $\Oex$ is the same as that of $\Omega$. Given a functor $G\:\C\to \D$ that preserves exact sequences, we define the functor $[G]\:\Ses{\C}\to \Ses{\D}$ to send a short exact sequence $X\ar{f} Y \ar{g} Z$ to the short exact sequence replacement of the exact sequence $G(X)\ar{G(f)} G(Y) \ar{G(g)} G(Z)$. Given then a morphism $(u,v,w)$ in $\Ses{\C}$ from $X\ar{f} Y \ar{g} Z$ to $X'\ar{f'} Y' \ar{g'} Z'$, we define $[G](u,v,w)$ to be the morphism $([G]^1(u,v,w),G(v),[G]^2(u,v,w))$ given as follows:
\begin{cd}[4]
    \& {K(G(g))} \\
	{G(X)} \&\& {G(Y)} \&\& {G(Z)} \\
	\&\&\& {C(G(f))} \\[-3ex]
	\& {K(G(g'))} \\
	{G(X')} \&\& {G(Y')} \&\& {G(Z')} \\
	\&\&\& {C(G(f'))}
	\arrow["{\ker{G(g)}}", from=1-2, to=2-3]
	\arrow[bend right,"{[G]^1(u,v,w)}"'{pos=0.55}, dashed, from=1-2, to=4-2]
	\arrow["{\exists ! \xi^1_G}", dashed, from=2-1, to=1-2]
	\arrow["{G(f)}"', from=2-1, to=2-3]
	\arrow["{G(u)}"', from=2-1, to=5-1]
	\arrow["{G(g)}", from=2-3, to=2-5]
	\arrow["{\coker{G(f)}}"', from=2-3, to=3-4]
	\arrow["{G(v)}", from=2-3, to=5-3]
	\arrow["{G(w)}", from=2-5, to=5-5]
	\arrow["{\exists !\xi^2_G}"', dashed, from=3-4, to=2-5]
	\arrow[bend left,"{[G]^2(u,v,w)}"{{pos=0.45}}, dashed, from=3-4, to=6-4]
	\arrow["{\ker{G(g')}}", from=4-2, to=5-3]
	\arrow["{\exists ! {\xi'}^1_G}", dashed, from=5-1, to=4-2]
	\arrow["{G(f')}"', from=5-1, to=5-3]
	\arrow["{G(g')}", from=5-3, to=5-5]
	\arrow["{\coker{G(f')}}"', from=5-3, to=6-4]
	\arrow["{\exists ! {\xi'}^2_G}"', dashed, from=6-4, to=5-5]
\end{cd}
The morphism $[G]^1(u,v,w)$ is the unique one induced by the universal property of $K(G(g'))$ starting from the null morphism $G(g')\c G(v)\c \ker{G(g)}=G(w)\c G(g)\c \ker{G(g)}$. Dually for $[G]^2(u,v,w)$. It is straightforward to prove that $[G]$ is a functor. Indeed, this is given by the fact that $[G]^1(u,v,w)$ and $[G]^2(u,v,w)$ are the unique morphisms that can fit in the relevant squares, since $\ker{G(g')}$ is mono and $\coker{G(f)}$ is epi.

Notice that when $G$ is actually a morphism in $\ClIdl$, and thus preserves short exact sequences, $[G]$ coincides with the functor $G\:\Ses{C}\to\Ses{\D}$ defined in \conx\ref{conscomonadbih}. This is also thanks to the choice we made in \remx\ref{remchoice}.

We now show that for every functor $G$ that preserves exact sequences, also $[G]$ preserves exact sequences, and is thus a morphism in $\ClIdlex$. It is straightforward to prove that a null sequence $(u',v',w')\c (u,v,w)$ of morphisms in $\Ses{\C}$ is an exact sequence if and only if the sequence $v'\c v$ in the middle column is an exact sequence and $K(u')\ar{\ker{u'}} X' \ar{} K(\rho)$ is a kernel and the dual holds, where $\rho$ is the unique appropriately induced morphism $C(v)\to C(w)$. It is then straightforward to prove that $[G]$ preserves exact sequences.

Finally, given a 2-cell $\alpha\:G\aR{}H\:\C\to \D$ in $\ClIdlex$, i.e.\ a natural transformation, we construct the natural transformation $[\alpha]\:[G]\aR{}[H]\:\Ses{\C}\to \Ses{\D}$ by setting its component on a short exact sequence $X\ar{f} Y \ar{g} Z$ to be the morphism $([\alpha]_{f,g}^1,\alpha_Y,[\alpha]_{f,g}^2)$ given as follows:
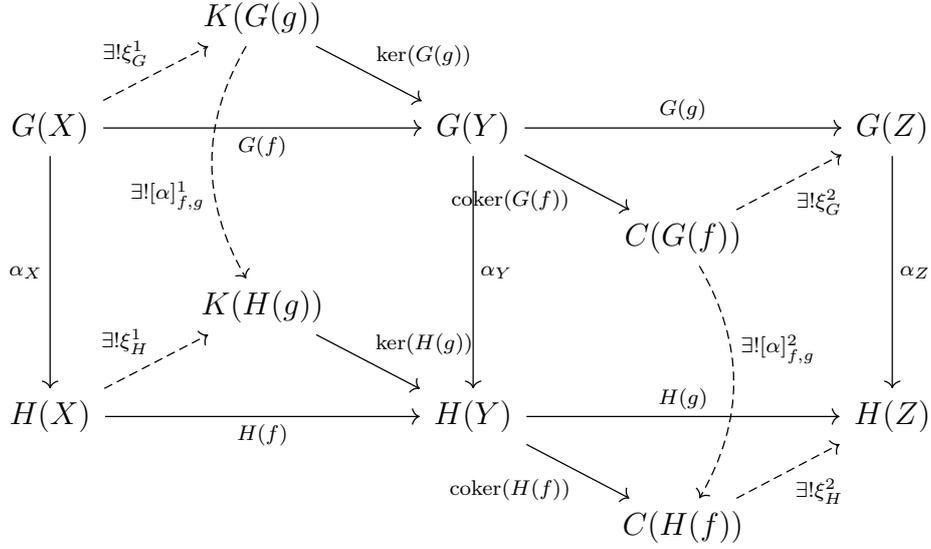
\begin{cd}[4]
	\& {K(G(g))} \\
	{G(X)} \&\& {G(Y)} \&\& {G(Z)} \\
	\&\&\& {C(G(f))} \\[-3ex]
	\& {K(H(g))} \\
	{H(X)} \&\& {H(Y)} \&\& {H(Z)} \\
	\&\&\& {C(H(f))}
	\arrow["{\ker{G(g)}}", from=1-2, to=2-3]
	\arrow[bend right,"{\exists ! [\alpha]_{f,g}^1}"'{pos=0.55}, dashed, from=1-2, to=4-2]
	\arrow["{\exists ! \xi^1_G}", dashed, from=2-1, to=1-2]
	\arrow["{G(f)}"', from=2-1, to=2-3]
	\arrow["{\alpha_X}"', from=2-1, to=5-1]
	\arrow["{G(g)}", from=2-3, to=2-5]
	\arrow["{\coker{G(f)}}"', from=2-3, to=3-4]
	\arrow["{\alpha_Y}", from=2-3, to=5-3]
	\arrow["{\alpha_Z}", from=2-5, to=5-5]
	\arrow["{\exists !\xi^2_G}"', dashed, from=3-4, to=2-5]
	\arrow[bend left,"{\exists ! [\alpha]_{f,g}^2}"{pos=0.45}, dashed, from=3-4, to=6-4]
	\arrow["{\ker{H(g)}}", from=4-2, to=5-3]
	\arrow["{\exists ! \xi^1_H}", dashed, from=5-1, to=4-2]
	\arrow["{H(f)}"', from=5-1, to=5-3]
	\arrow["{H(g)}", from=5-3, to=5-5]
	\arrow["{\coker{H(f)}}"', from=5-3, to=6-4]
	\arrow["{\exists !\xi^2_H}"', dashed, from=6-4, to=5-5]
\end{cd}
The morphism $[\alpha]_{f,g}^1$ is the unique one induced by the universal property of the kernel $K(H(g))$ starting from the null morphism
$$H(g)\c \alpha_Y \c \ker{G(g)}=\alpha_Z\c G(g)\c \ker{G(g)}.$$
Dually for $[\alpha]_{f,g}^2$. It is straightforward to prove that $[\alpha]$ is a natural transformation. Indeed, this is given by the fact that $[\alpha]_{f,g}^1$ and $[\alpha]_{f,g}^2$ are the unique morphisms that can fit in the relevant squares, since $\ker{H(g)}$ is mono and $\coker{G(f)}$ is epi.

It is then straightforward to prove that $\Oex$ is a (normal) pseudofunctor. The fact that it preserves identities is easy to show, thanks to choice we made in \remx\ref{remchoice}. Consider then $\C\ar{G}\D\ar{H}\E$ in $\ClIdlex$. We construct a natural isomorphism
\begin{cd}[4]
    \& {\Ses{\D}} \\
	{\Ses{\C}} \&\& {\Ses{\E}}
	\arrow["{[H]}", from=1-2, to=2-3]
	\arrow["{[G]}", from=2-1, to=1-2]
	\arrow["{[H\c G]}"', from=2-1, to=2-3]
	\arrow["\beta",shift left=2ex, iso, from=2-1, to=2-3]
\end{cd}
Given a short exact sequence $X\ar{f} Y \ar{g} Z$ in $\C$, consider the diagram
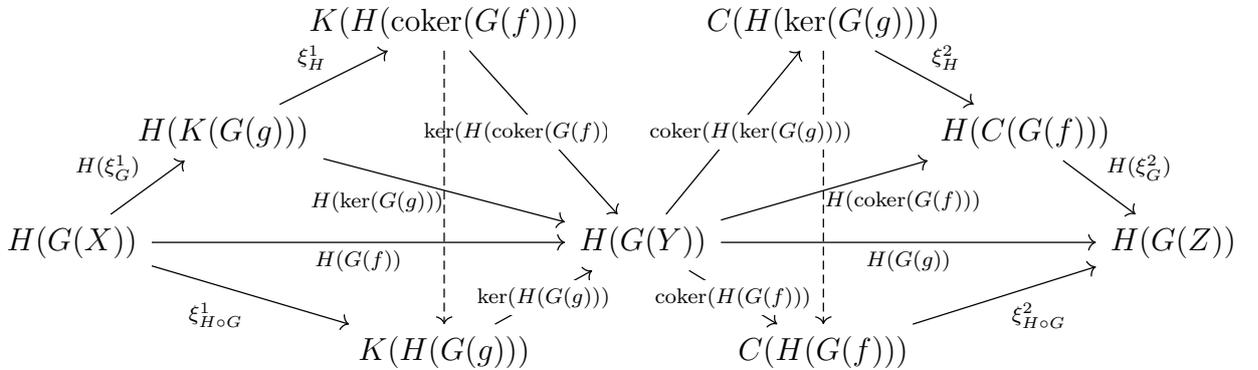
\begin{cd}[4][-2]
	\&\& {K(H(\coker{G(f)}))} \&\& {C(H(\ker{G(g)}))} \\
	\& {H(K(G(g)))} \&\&\&\& {H(C(G(f)))} \\
	{H(G(X))} \&\&\& {H(G(Y))} \&\&\& {H(G(Z))} \\
	\&\& {K(H(G(g)))} \&\& {C(H(G(f)))}
	\arrow["{\ker{H(\coker{G(f)})}\quad\h[3]}"{description}, from=1-3, to=3-4]
	\arrow[dashed, from=1-3, to=4-3]
	\arrow["{\xi^2_H}", from=1-5, to=2-6]
	\arrow[dashed, from=1-5, to=4-5]
	\arrow["{\xi^1_H}", from=2-2, to=1-3]
	\arrow["{H(\ker{G(g)})}"'{inner sep=0 ex}, from=2-2, to=3-4]
	\arrow["{H(\xi^2_G)}", from=2-6, to=3-7]
	\arrow["{H(\xi^1_G)}", from=3-1, to=2-2]
	\arrow["{H(G(f))}"', from=3-1, to=3-4]
	\arrow["{\xi^1_{H\c G}}"', from=3-1, to=4-3]
	\arrow["{\h[3]\quad\coker{H(\ker{G(g)})}}"{description}, from=3-4, to=1-5]
	\arrow["{H(\coker{G(f)})}"'{inner sep=0 ex}, from=3-4, to=2-6]
	\arrow["{H(G(g))}"', from=3-4, to=3-7]
	\arrow["{\coker{H(G(f))}}"'{description}, from=3-4, to=4-5]
	\arrow["{\ker{H(G(g))}}"'{description}, from=4-3, to=3-4]
	\arrow["{\xi^2_{H\c G}}"', from=4-5, to=3-7]
\end{cd}
We define the component $\beta_{f,g}$ of $\beta$ to be given by the dashed morphisms in the diagram above, that are induced by the universal property of the relevant kernels and cokernels. This is thanks to the fact that $H(G(g))\c \ker{H(\coker{G(f)})}$ is a null morphism, and the dual fact. The components of $\beta$ are isomorphisms because we can also induce the inverses of the dashed morphisms in the diagram above. This heavily relies on the fact that $\xi^2_G$ and thus $H(\xi^2_G)$ reflect null morphisms, and the dual fact. It is straightforward to prove that $\beta$ is a natural isomorphism, and that $\Oex$ is a pseudo-functor. 
\end{construction}

\begin{theorem}\label{theoromegaex}
    The pseudofunctor
    \begin{fun}
	\Oex & \: & \ClIdlex & \too & \ClIdlex \\[0.3ex]
    && \tcv*{\C}{\D}{G}{H}{\alpha} & \mto & \tcv*{\Ses{\C}}{\Ses{\D}}{[G]}{[H]}{[\alpha]}
\end{fun}
of \conx\ref{consomegaex} extends to a pseudo-comonad $\Oex$ on the 2-category $\ClIdlex$ of categories equipped with a closed ideal that have all (relative) kernels and cokernels, and functors that preserve exact sequences.

Moreover, all pretorsion theories (on semi-exact categories) are pseudo-coalgebras for $\Oex$. 
\end{theorem}
\begin{proof}
    We define the counit $\epsilon^{\operatorname{ex}}\:\Oex\aR{} \Id{}$ to have general component on $\C\in \ClIdlex$ given by the same functor $\epsilon_{\C}$ of \thex\ref{theorbih}. Since $\epsilon_{\C}$ is a morphism in $\ClIdl$, it is also a morphism in $\ClIdlex$. Moreover, this extension $\epsilon^{\operatorname{ex}}$ of $\epsilon$ remains 2-natural, as it is easy to see. Similarly, we then define the comultiplication $\delta^{\operatorname{ex}}$ to have general component on $\C\in \ClIdlex$ given by the same functor $\delta_{\C}$ of \thex\ref{theoromegaex}. Since $\delta_{\C}$ is a morphism in $\ClIdl$, it is also a morphism in $\ClIdlex$. Moreover, it is straightforward to prove that this extension $\delta^{\operatorname{ex}}$ of $\delta$ remains pseudo-natural. 

    Since the components $\delta_{\C}$ of $\delta$ preserve kernels and cokernels, and thus preserve short exact sequences, we have that $[\delta_{\C}]$ just applies $\delta_{\C}$ without the need to calculate a replacement. So the invertible modification $\Xi^{\operatorname{ex}}$ that extends the $\Xi$ of \thex\ref{theorbih} can have the same components of $\Xi$. It is then straightforward to prove that $\Xi^{\operatorname{ex}}$ is indeed a modification and that $(\Oex,\delta^{\operatorname{ex}},\epsilon^{\operatorname{ex}},\Xi^{\operatorname{ex}})$ is a pseudo-comonad. 

    A pseudo-coalgebra for $\Oex$ is $(\C,\N)\in \ClIdlex$ equipped with a coalgebra morphism $\lambda\:\C\to \Ses{\C}$ in $\ClIdlex$ and a natural isomorphism 
    \begin{cd}
        \C \& {\Ses{\C}} \\
	{\Ses{\C}} \& {\Ses{\Ses{\C}}}
	\arrow["\lambda", from=1-1, to=1-2]
	\arrow[""{name=0, anchor=center, inner sep=0}, "\lambda"', from=1-1, to=2-1]
	\arrow[""{name=1, anchor=center, inner sep=0}, "{[\lambda]}", from=1-2, to=2-2]
	\arrow["{\delta_{\C}}"', from=2-1, to=2-2]
	\arrow["{\lambda_{\delta}}"'{inner sep = 1ex}, iso, from=0, to=1]
    \end{cd}
    such that the triangle
    \begin{cd}[5]
        \C \& {\Ses{\C}} \\
	\& \C
	\arrow["\lambda", from=1-1, to=1-2]
	\arrow[bend right=20, equals, from=1-1, to=2-2]
	\arrow["{\epsilon_{\C}}", from=1-2, to=2-2]
    \end{cd}
    commutes and the axioms of pseudo-coalgebra are satisfied. Thanks to the triangle above, the coalgebra map $\lambda$ thus assigns to every $X\in \C$ a short exact sequence
    $$T^X\ar{\ell^x} X \ar{r^X}F^X$$
    with $X$ in the middle, and to every morphism $h\:X\to Y$ in $\C$ a morphism in $\Ses{\C}$
    \begin{cd}[5]
        {T^X} \& X \& {F^X} \\
	{T^Y} \& Y \& {F^Y}
	\arrow["{\ell^X}", from=1-1, to=1-2]
	\arrow["{h^T}"', from=1-1, to=2-1]
	\arrow["{r^X}", from=1-2, to=1-3]
	\arrow["h", from=1-2, to=2-2]
	\arrow["{h^F}", from=1-3, to=2-3]
	\arrow["{\ell^Y}"', from=2-1, to=2-2]
	\arrow["{r^Y}"', from=2-2, to=2-3]
    \end{cd}
    with $h$ in the middle.

    Moreover, for every $X\in \C$, $\lambda_{\delta}$ provides an isomorphism in $\Ses{\Ses{\C}}$ given by
    \begin{cd}[4][5]
        {T^{T^X}} \&[-5ex]\&[-2ex] {T^X} \&[-3ex]\&[-2ex] {F^{T^X}} \&[-5ex]\&[-5ex] {T^X} \& {T^X} \& {C(\id{T^X})} \\
	\& {K((r^X)^T)} \&\& {T^X} \&\& {C(\gamma)} \& {T^X} \& X \& {F^X} \\
	{T^X} \&\& X \&\& {F^X} \&\& {K(\id{F^X})} \& {F^X} \& {F^{F^X}} \\
	\& {K(\omega)} \&\& {F^X} \&\& {C((\ell^X)^F)} \\
	{T^{F^X}} \&\& {F^X} \&\& {F^{F^X}}
	\arrow["{\ell^{T^X}}", from=1-1, to=1-3]
	\arrow[from=1-1, to=2-2]
	\arrow["{(\ell^X)^T}"', from=1-1, to=3-1]
	\arrow["{r^{T^X}}", from=1-3, to=1-5]
	\arrow[equals, from=1-3, to=2-4]
	\arrow["{\ell^X}", from=1-3, to=3-3]
	\arrow[from=1-5, to=2-6]
	\arrow["{(\ell^X)^F}", from=1-5, to=3-5]
	\arrow[equals, from=1-7, to=1-8]
	\arrow[equals, from=1-7, to=2-7]
	\arrow["{\coker{\id{T^X}}}", from=1-8, to=1-9]
	\arrow["{\ell^X}", from=1-8, to=2-8]
	\arrow["w", from=1-9, to=2-9]
	\arrow["{\lambda_{\delta}^{1,1}}"{description}, from=2-2, to=1-7]
	\arrow["\gamma", from=2-2, to=2-4]
	\arrow[from=2-2, to=3-1]
	\arrow[from=2-4, to=2-6]
	\arrow["{\ell^X}", from=2-4, to=3-3]
	\arrow["{\lambda_{\delta}^{1,3}}"{description}, from=2-6, to=1-9]
	\arrow[from=2-6, to=3-5]
	\arrow["{\ell^X}", from=2-7, to=2-8]
	\arrow["u"', from=2-7, to=3-7]
	\arrow["{r^X}", from=2-8, to=2-9]
	\arrow["{r^X}"', from=2-8, to=3-8]
	\arrow[equals, from=2-9, to=3-9]
	\arrow["{\ell^X}", from=3-1, to=3-3]
	\arrow[from=3-1, to=4-2]
	\arrow["{(r^X)^T}"', from=3-1, to=5-1]
	\arrow["{r^X}", from=3-3, to=3-5]
	\arrow["{r^X}", from=3-3, to=4-4]
	\arrow["{r^X}"', from=3-3, to=5-3]
	\arrow[from=3-5, to=4-6]
	\arrow["{(r^X)^T}", from=3-5, to=5-5]
	\arrow["{\ker{\id{F^X}}}"', from=3-7, to=3-8]
	\arrow[equals, from=3-8, to=3-9]
	\arrow["{\lambda_{\delta}^{3,1}}"{description}, from=4-2, to=3-7]
	\arrow[from=4-2, to=4-4]
	\arrow[from=4-2, to=5-1]
	\arrow["\omega"', from=4-4, to=4-6]
	\arrow[equals, from=4-4, to=5-3]
	\arrow["{\lambda_{\delta}^{3,3}}"{description}, from=4-6, to=3-9]
	\arrow[from=4-6, to=5-5]
	\arrow["{\ell^{F^X}}"', from=5-1, to=5-3]
	\arrow["{r^{F^X}}"', from=5-3, to=5-5]
    \end{cd}

    We now prove that all pretorsion theories $(\C,\T,\F)$ yield pseudo-coalgebras for $\Oex$. Of course, $(\C,\T\cap \F)\in \ClIdlex$. We define the coalgebra morphism to be the same functor of \thex\ref{theorbih}
    \begin{fun}
	\lambda & \: & \C \hphantom{C} & \too & \hphantom{CC}\Ses{\C} \\[0.3ex]
    && \fib{X}{h}{Y} & \mto & \begin{cd}*[4][5.5]
{T^X} \& X \& {F^X} \\
	{T^Y} \& Y \& {F^Y}
	\arrow["{\ell^X}", from=1-1, to=1-2]
	\arrow["{h^T}"', from=1-1, to=2-1]
	\arrow["{r^X}", from=1-2, to=1-3]
	\arrow["h", from=1-2, to=2-2]
	\arrow["{h^F}", from=1-3, to=2-3]
	\arrow["{\ell^Y}"', from=2-1, to=2-2]
	\arrow["{r^Y}"', from=2-2, to=2-3]
    \end{cd}
    \end{fun}
    using the chosen short exact sequences and the torsion and torsion-free parts of morphisms given by the pretorsion theory. It is straightforward to show that $\lambda$ preserves exact sequences, using the characterization of exact sequences in $\Ses{\C}$ that we presented in \conx\ref{consomegaex}. 
    It is then straightforward to modify the natural isomorphism $\lambda_{\delta}$ that we constructed in \thex\ref{theorbih} to yield a pseudo-coalgebra $(\C,\T\cap \F,\lambda,\lambda_{\delta})$ for $\Oex$. 
\end{proof}

\begin{remark}
    The pseudo-coalgebras for $\Oex$ correspond to a generalization of pretorsion theories. Indeed, given a pseudo-coalgebra, its coalgebra map 
 \begin{fun}
	{\lambda} & \: & \C & \too & \Ses{\C} \\[1ex]
    && \fib{X}{h}{Y} & \mto &  
\begin{cd}*[4.5]
{T^X} \arrow[r,"{\ell^X}"] \arrow[d,"{h^T}"'] \& {X} \arrow[d,"{h}"] \arrow[r,"{r^X}"] \& {F^X} \arrow[d, "{h^F}"]\\
{T^Y} \arrow[r,"{\ell^Y}"'] \& {Y} \arrow[r,"{r^Y}"']\& {F^Y} 
\end{cd}
\end{fun}
is such that $(\ell^X)^F$ null and $(r^X)^T$ null for every $X\in \C$. Intuitively, this means that we do not necessarily associate to every object $X\in \C$ a short exact sequence involving a torsion object and a torsion-free object. Instead, we associate to $X$ a short exact sequence involving a torsion and a torsion-free morphism. 
We plan to further investigate these generalized pretorsion theories in future work.
\end{remark}

In the following, we describe a property that determines pretorsion theories among the pseudo-coalgebras for $\Oex$.

\begin{proposition}
    Pretorsion theories are precisely the pseudo-coalgebras for $\Oex$ such that the coalgebra map ${\lambda\: \C \to \Ses{\C}}$ preserves the short exact sequences in its image.
\end{proposition}
\begin{proof}
    If $\lambda$ preserves the short exact sequences in its image, then the composite $\C\ar{\lambda}\Ses{\C}\ar{[\lambda]}\Ses{\Ses{\C}}$ coincides with $\C\ar{\lambda}\Ses{\C}\ar{\lambda}\Ses{\Ses{\C}}$ as defined in \thex\ref{theorbih}, thanks to the choice we made in \remx\ref{remchoice} (see also \conx\ref{consomegaex}). So the natural isomorphism $\lambda_{\delta}$ guarantees that
    $$T^{T^X}\aiso{\lambda_{\delta}^{1,1}} T^X \quad \text{ and } \quad F^{T^X} \text{ null object}$$
    $$F^{F^X}\aiso{\lambda_{\delta}^{3,3}} F^X \quad \text{ and } \quad T^{F^X} \text{ null object}$$
    And we obtain a pretorsion theory. 

    It is straightforward to prove the converse, using that the short exact sequences in the image of $\lambda$ have a torsion object on the left and a torsion-free object on the right. 
\end{proof}

\begin{remark}
We observe that in this setting we have the following quintuple of adjoint functors:
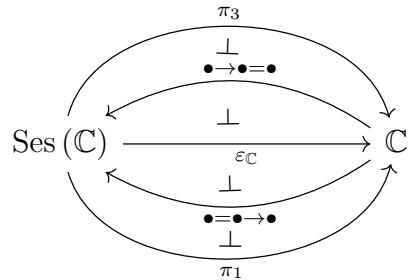
\begin{cd}[18][18]
{\Ses{\C}} \arrow[r,"{\varepsilon_{\C}}"'{name=C}] \arrow[r,"{\pi_3}"{name=A}, bend left=75] \arrow[r,"{\pi_1}"', " "{pos=0.556, name=E}, bend right=75] \& {\C} \arrow[l,"{\bullet = \bullet \to \bullet}", " "{name=D, pos=0.495}, bend left=32]  \arrow[l,"{\bullet \to \bullet = \bullet}"',""'{name=B,pos=0.55},""{name=F,pos=0.45}, bend right=32]  \arrow[from=A,to=B, "{}", adj] \arrow[from=F,to=C, "{}", adj, shift right=1.3ex, yshift=0.31ex] \arrow[from=C,to=D, "{}", adj, shift right=1.3ex,yshift=0.51ex] \arrow[from=D,to=E, "{}", adj, shift right=1.3ex, yshift=-0.51ex] 
\end{cd}
With some work, it is possible to characterize pretorsion theories among the pseudo-coalgebras for $\Oex$ in terms of equations that only involve these adjoints. 

In future work, we plan to further analyze the properties of general comonads that can be equipped with a diagram of five adjoints like that of the picture above.
\end{remark}




\bibliographystyle{abbrv}
\bibliography{bibliography}


\end{document}